\documentclass[a4paper,11pt,fleqn]{article}
\usepackage{pstricks}
\usepackage{amsmath}
\usepackage{amssymb}
\usepackage{pifont}
\usepackage{theorem} 
\usepackage{euscript}
\usepackage{exscale,relsize}
\usepackage{epic,eepic}
\usepackage{multicol}
\usepackage{makeidx}
\usepackage{srcltx}
\usepackage{xcolor}
\usepackage{url}
\usepackage{charter}
\usepackage{booktabs,subfig,graphicx}
\newcommand{\email}[1]{\href{mailto:#1}{\nolinkurl{#1}}}
\usepackage[normalem]{ulem}     
\newlength{\mySubFigSize}
\newcommand{\minimize}[2]{\ensuremath{\underset{\substack{{#1}}}%
{\text{minimize}}\;\;#2 }}

\newcommand{\scal}[2]{{\left\langle{{#1}\mid{#2}}\right\rangle}}

\newcommand{\menge}[2]{\big\{{#1}~\big |~{#2}\big\}} 
 
\newcommand{\HHH}{{\ensuremath{\boldsymbol{\mathcal H}}}}
\newcommand{\KKK}{\ensuremath{\boldsymbol{\mathcal K}}}
\newcommand{\GGG}{\ensuremath{{\boldsymbol{\mathcal G}}}}

\newcommand{\HH}{\ensuremath{{\mathcal H}}}
\newcommand{\VEC}{\text{vec}}
\newcommand{\GG}{\ensuremath{{\mathcal G}}}
\newcommand{\KK}{\ensuremath{{\mathcal K}}}
\newcommand{\Sum}{\ensuremath{\displaystyle\sum}}
\newcommand{\emp}{\ensuremath{{\varnothing}}}

\newcommand{\Id}{\ensuremath{\operatorname{Id}}\,}

\newcommand{\RR}{\ensuremath{\mathbb{R}}}
\newcommand{\RP}{\ensuremath{\left[0,+\infty\right[}}

\newcommand{\BL}{\ensuremath{\EuScript B}\,}
\newcommand{\RPP}{\ensuremath{\left]0,+\infty\right[}}

\newcommand{\RPX}{\ensuremath{\left[0,+\infty\right]}}
\newcommand{\RX}{\ensuremath{\left]-\infty,+\infty\right]}}
\newcommand{\RXX}{\ensuremath{\left[-\infty,+\infty\right]}}

\newcommand{\NN}{\ensuremath{\mathbb N}}

\newcommand{\weakly}{\ensuremath{\:\rightharpoonup\:}}
\newcommand{\exi}{\ensuremath{\exists\,}}
\newcommand{\ran}{\ensuremath{\text{\rm ran}\,}}

\newcommand{\pinf}{\ensuremath{{+\infty}}}

\newcommand{\dom}{\ensuremath{\text{\rm dom}\,}}
\newcommand{\prox}{\ensuremath{\text{\rm prox}}}

\newcommand{\gra}{\ensuremath{\text{gra}\,}}

\newcommand{\sri}{\ensuremath{\text{\rm sri}\,}}
\newcommand{\reli}{\ensuremath{\text{\rm ri}\,}}
\newcommand{\infconv}{\ensuremath{\mbox{\small$\,\square\,$}}}
\newcommand{\pushfwd}{\ensuremath{\mbox{\Large$\,\triangleright\,$}}}
\newcommand{\zeroun}{\ensuremath{\left]0,1\right[}}

\renewcommand{\leq}{\ensuremath{\leqslant}}
\renewcommand{\geq}{\ensuremath{\geqslant}}

\newcommand{\obs}{y}
\newcommand{\tv}{\text{tv}}     

\newcommand{\TV}{D^{(1)}}       
\newcommand{\higherTV}{D^{(2)}}
\newcommand{\Dh}{D_\leftrightarrow} 
\newcommand{\Dv}{D_\updownarrow}
\newcommand{\Dhh}{\widetilde{D}_\leftrightarrow}
\newcommand{\Dvv}{\widetilde{D}_\updownarrow}

\newtheorem{theorem}{Theorem}[section]
\newtheorem{lemma}[theorem]{Lemma}
\newtheorem{corollary}[theorem]{Corollary}
\newtheorem{proposition}[theorem]{Proposition}

\theoremstyle{plain}{\theorembodyfont{\rmfamily}%
}
\theoremstyle{plain}{\theorembodyfont{\rmfamily}%
}
\theoremstyle{plain}{\theorembodyfont{\rmfamily}%
\newtheorem{remark}[theorem]{Remark}}
\theoremstyle{plain}{\theorembodyfont{\rmfamily}%
}
\theoremstyle{plain}{\theorembodyfont{\rmfamily}%
}
\theoremstyle{plain}{\theorembodyfont{\rmfamily}%
\newtheorem{definition}[theorem]{Definition}}
\theoremstyle{plain}{\theorembodyfont{\rmfamily}%
\newtheorem{problem}[theorem]{Problem}}

\numberwithin{equation}{section}
\begin{document}

\title{\sffamily\huge An Algorithm for Splitting Parallel Sums 
of Linearly Composed Monotone Operators, with Applications to
Signal Recovery}

\author{Stephen Becker and Patrick L. Combettes\\[5mm]
\small UPMC Universit\'e Paris 06\\
\small Laboratoire Jacques-Louis Lions -- UMR CNRS 7598\\
\small 75005 Paris, France\\
\small \ttfamily{becker@ljll.math.upmc.fr}, 
\ttfamily{plc@math.jussieu.fr}\\[4mm]
}

\date{~}

\maketitle

\vskip 8mm

\begin{abstract} 
\noindent
We present a new primal-dual splitting algorithm for structured
monotone inclusions in Hilbert spaces and analyze its asymptotic
behavior. A novelty of our framework, which is motivated by 
image recovery applications, is to consider inclusions that
combine a variety of monotonicity-preserving operations such as 
sums, linear compositions, parallel sums, and a new notion of 
parallel composition. The special case of minimization problems is
studied in detail, and applications to signal recovery 
are discussed. Numerical simulations are 
provided to illustrate the implementation of the algorithm.
\end{abstract} 

{\bfseries Keywords} 
duality,
image recovery,
infimal convolution,
monotone operator,
parallel composition,
parallel sum,
proximity operator,
splitting algorithm

{\bfseries Mathematics Subject Classifications (2010)} 
Primary 47H05; Secondary 65K05, 90C25.

\maketitle

\newpage

\section{Introduction}

Let $A$ and $B$ be set-valued monotone operators acting on
a real Hilbert space $\HH$. The first operator splitting 
algorithms were developed in the late 1970s to solve inclusion 
problems of the form 
\begin{equation}
\label{e:1979}
\text{find}\;\;\overline{x}\in\HH\;\;\text{such that}\;\;
0\in A\overline{x}+B\overline{x},
\end{equation}
by using separate applications of the operators $A$ and $B$ 
at each iteration; see
\cite{Opti04,Ecks92,Lion79,Merc79,Tsen91,Tsen00} and the 
references therein. Because of increasingly complex problem
formulations, more sophisticated splitting algorithm have recently
arisen. Thus, the splitting method proposed in 
\cite{Siop11} can solve problems of the type
\begin{equation}
\label{e:2011}
\text{find}\;\;\overline{x}\in\HH\;\;\text{such that}\;\;
0\in A\overline{x}+\sum_{k=1}^r\big(L_k^*\circ B_k\circ
L_k\big)\overline{x},
\end{equation}
where each monotone operator $B_k$ acts on a real Hilbert space
$\GG_k$ and each $L_k$ is a bounded linear operator from $\HH$ to
$\GG_k$. This model was further refined in \cite{Svva12}
by considering inclusions of the form 
\begin{equation}
\label{e:2012}
\text{find}\;\;\overline{x}\in\HH\;\;\text{such that}\;\;
0\in A\overline{x}+\sum_{k=1}^r\big(L_k^*\circ(B_k\infconv D_k)
\circ L_k\big)\overline{x}+C\overline{x},
\end{equation}
where $D_k$ is a monotone operator acting on $\GG_k$
such that $D_k^{-1}$ is Lipschitzian,
\begin{equation} 
\label{e:parsumdef}
B_k\infconv D_k=\big(B_k^{-1}+D_k^{-1}\big)^{-1}
\end{equation}
is the parallel sum of $B_k$ and $D_k$, and $C\colon\HH\to\HH$
is a Lipschitzian monotone operator. More recent developments 
concerning splitting methods for models featuring parallel sums 
can be found in \cite{Botr2013,Botr2014,Bang13}. 
In the present paper, motivated by variational problems arising 
in image recovery, we consider a new type of inclusions that 
involve both parallel sum and ``parallel composition'' operations 
in the sense we introduce below.

\begin{definition}
\label{d:2011-01-05}
Let $\HH$ and $\GG$ be real Hilbert spaces, let 
$A\colon\HH\to 2^{\HH}$, and let $L\in\BL(\HH,\GG)$. Then
the \emph{parallel composition} of $A$ by $L$ is 
\begin{equation}
\label{e:2011-01-05a}
L\pushfwd A=\big(L\circ A^{-1}\circ L^*)^{-1}.
\end{equation}
\end{definition}

The primal-dual inclusion problem under consideration will be the
following (our notation is standard, see Section~\ref{sec:21} for
details).

\begin{problem}
\label{prob:1}
Let $\HH$ be a real Hilbert space, let $r$ be a strictly positive
integer, let $z\in\HH$, let $A\colon\HH\to 2^{\HH}$ be maximally 
monotone, and let $C\colon\HH\to\HH$ be monotone and 
$\mu$-Lipschitzian for some $\mu\in\RP$. 
For every integer $k\in\{1,\ldots,r\}$, let $\GG_k$ and $\KK_k$
be real Hilbert spaces, let 
$B_k\colon\GG_k\to 2^{\GG_k}$ and $D_k\colon\KK_k\to 2^{\KK_k}$ 
be maximally monotone, and let $L_k\in\BL(\HH,\GG_k)$ and
$M_k\in\BL(\HH,\KK_k)$. It is assumed that
\begin{equation}
\label{e:2013-04-30b}
\beta=\mu+\sqrt{\sum_{k=1}^r\|L_k\|^2+\max_{1\leq k\leq r}
\big(\|L_k\|^2+\|M_k\|^2\big)}\:>0
\end{equation}
and that the inclusion
\begin{equation}
\label{e:2013-04-30p}
\text{find}\;\;\overline{x}\in\HH\;\;\text{such that}\;\;
z\in A\overline{x}+\sum_{k=1}^r\big((L_k^*\circ B_k\circ L_k)
\infconv (M_k^*\circ D_k\circ M_k)\big)\overline{x}+C\overline{x}
\end{equation}
possesses at least one solution. Solve \eqref{e:2013-04-30p} 
together with the dual problem
\begin{multline}
\label{e:2013-04-30d}
\text{find}\;\;\overline{v_1}\in\GG_1,\:\ldots,\:
\overline{v_r}\in\GG_r\;\:\text{such that}\;\;
(\forall k\in\{1,\ldots,r\})\\
0\in-L_k\bigg((A+C)^{-1}\bigg(z-\sum_{l=1}^rL_{l}^*
\overline{v_l}\bigg)\bigg)+B_k^{-1}\overline{v_k}
+L_k\Big(\big(M_k^*\pushfwd D_k^{-1}\big)
(L_k^*\overline{v_k})\Big).
\end{multline}
\end{problem}

The paper is organized as follows. In Section~\ref{sec:2} we
define our notation and provide preliminary results. 
In particular,
we establish some basic properties of the parallel composition
operation introduced in Definition~\ref{d:2011-01-05} and discuss
an algorithm recently proposed in \cite{Siop13} that will
serve as a basis for our splitting method. In Section~\ref{sec:3},
our algorithm is presented and weak and strong convergence results
are established. Section~\ref{sec:4} is devoted to the application
of this algorithm to convex minimization problems. Finally, in
Section~\ref{sec:5}, we propose applications of the results of
Section~\ref{sec:4} to a concrete problem in image recovery,
along with numerical results.
 
\section{Notation and preliminary results}
\label{sec:2}

\subsection{Notation and definitions}
\label{sec:21}
The following notation will be used throughout.
$\HH$, $\GG$, and $\KK$ are real Hilbert spaces.
We denote the scalar product of a Hilbert space by 
$\scal{\cdot}{\cdot}$ and the associated norm by $\|\cdot\|$.
The symbols $\weakly$ and $\to$ denote, respectively, weak and 
strong convergence. $\BL(\HH,\GG)$ is the space of bounded 
linear operators from $\HH$ to $\GG$. The Hilbert direct sum 
of $\HH$ and $\GG$ is denoted by $\HH\oplus\GG$.
Given two sequence $(x_n)_{n\in\NN}$ and 
$(y_n)_{n\in\NN}$ in $\HH$, it will be convenient to
use the notation
\begin{equation}
\label{e:jeddah2013-05-23a}
\big[(\forall n\in\NN)\;\: x_n\approx y_n\big]
\quad\Leftrightarrow\quad\sum_{n\in\NN}\|x_n-y_n\|<\pinf
\end{equation}
to model the tolerance to errors in the implementation of the
algorithms. 

The power set of $\HH$ is denoted by $2^\HH$.
Let $A\colon\HH\to 2^{\HH}$ be a set-valued operator.
We denote by $\ran A=\menge{u\in\HH}{(\exi x\in\HH)\;u\in Ax}$ 
the range of $A$, by $\dom A=\menge{x\in\HH}{Ax\neq\emp}$ the
domain of $A$, by $\gra A=\menge{(x,u)\in\HH\times\HH}{u\in Ax}$ 
the graph of $A$, and by $A^{-1}$ the inverse of $A$, i.e., 
the set-valued operator with graph
$\menge{(u,x)\in\HH\times\HH}{u\in Ax}$. The resolvent of $A$ is
$J_A=(\Id+A)^{-1}$. Moreover, $A$ is monotone if
\begin{equation}
(\forall (x,y)\in\HH\times\HH)(\forall(u,v)\in Ax\times Ay)\quad
\scal{x-y}{u-v}\geq 0,
\end{equation}
and maximally monotone if there exists no monotone operator 
$B\colon\HH\to 2^{\HH}$ such that $\gra A\subset\gra B\neq\gra A$.
In this case, $J_A$ is a single-valued, nonexpansive operator 
defined everywhere in $\HH$. 
We say that $A$ is uniformly monotone at $x\in\dom A$ if there 
exists an increasing function $\phi\colon\RP\to\RPX$ that vanishes 
only at $0$ such that
\begin{equation}
\label{e:Unifmon}
(\forall u\in Ax)(\forall (y,v)\in\gra A)\quad
\scal{x-y}{u-v}\geq\phi(\|x-y\|).
\end{equation}
We denote by $\Gamma_0(\HH)$ the class of lower semicontinuous 
convex functions $f\colon\HH\to\RX$ such that
$\dom f=\menge{x\in\HH}{f(x)<\pinf}\neq\emp$. Let
$f\in\Gamma_0(\HH)$. The conjugate of $f$ is the function 
$f^*\in\Gamma_0(\HH)$ defined by 
$f^*\colon u\mapsto\text{sup}_{x\in\HH}(\scal{x}{u}-f(x))$. 
For every $x\in\HH$, $f+\|x-\cdot\|^2/2$ possesses a 
unique minimizer, which is denoted by $\prox_fx$. The operator
$\prox_f$ can also be defined as a resolvent, namely
\begin{equation}
\label{e:prox2}
\prox_f=(\Id+\partial f)^{-1}=J_{\partial f},
\end{equation}
where $\partial f\colon\HH\to 2^{\HH}\colon x\mapsto
\menge{u\in\HH}{(\forall y\in\HH)\;\:\scal{y-x}{u}+f(x)\leq f(y)}$ 
is the subdifferential of $f$, which is maximally monotone.
We say that $f$ uniformly convex at $x\in\dom f$ if there exists 
an increasing function $\phi\colon\RP\to\RPX$ that vanishes only 
at $0$ such that
\begin{equation}
\label{e:Unifconvex}
(\forall y\in\dom f)(\forall\alpha\in\zeroun)\quad
f(\alpha x+(1-\alpha)y)+\alpha(1-\alpha)\phi(\|x-y\|)\leq
\alpha f(x)+(1-\alpha)f(y).
\end{equation}
The infimal convolution of two functions $f_1$ and 
$f_2$ from $\HH$ to $\RX$ is
\begin{equation}
\label{e:J1}
f_1\infconv f_2\colon\HH\to\RXX\colon x\mapsto
\inf_{y\in\HH}\big(f_1(x-y)+f_2(y)\big),
\end{equation}
and the infimal postcomposition of $f\colon\HH\to\RXX$
by $L\colon\HH\to\GG$ is
\begin{equation}
\label{e:J2}
L\pushfwd f\colon\GG\to\RXX\colon y\mapsto
\inf f\big(L^{-1}\{y\}\big)=\inf_{\substack{x\in\HH\\ Lx=y}}f(x).
\end{equation}
Let $C$ be a convex subset of $\HH$. 
The indicator function of $C$ is denoted by $\iota_C$, 
and the strong relative interior of $C$, i.e.,
the set of points $x\in C$ such that the cone generated by 
$-x+C$ is a closed vector subspace of $\HH$, by $\sri C$.

For a detailed account of the above concepts, see \cite{Livre1}.

\subsection{Parallel composition}

In this section we explore some basic properties of the parallel
composition operation introduced in Definition~\ref{d:2011-01-05}
which are of interest in their own right. First, we justify the 
terminology via the following connection with the parallel sum.

\begin{lemma}
\label{l:3}
Let $A\colon\HH\to 2^{\HH}$, let $B\colon\HH\to 2^{\HH}$, and let 
$L\colon\HH\oplus\HH\to\HH\colon (x,y)\mapsto x+y$.
Then $L\pushfwd(A\times B)=A\infconv B$.
\end{lemma}
\begin{proof}
Since $L^*\colon\HH\to\HH\oplus\HH\colon x\mapsto(x,x)$, the
announced identity is an immediate consequence of 
\eqref{e:parsumdef} and \eqref{e:2011-01-05a}.
\end{proof}

\begin{lemma}
\label{l:2}
Let $A\colon\HH\to 2^{\HH}$, let $B\colon\GG\to 2^{\GG}$,
and let $L\in\BL(\HH,\GG)$. Then the following hold.
\begin{enumerate}
\item
\label{l:2i}
$((L\pushfwd A)\infconv B)^{-1}=L\circ A^{-1}\circ L^*+B^{-1}$.
\item
\label{l:2ii}
Suppose that $A$ and $B$ are monotone. Then 
$(L\pushfwd A)\infconv B$ is monotone.
\item
\label{l:2iii}
Suppose that $A$ and $B$ are maximally monotone and that the 
cone generated by $L^*(\ran B)-\ran A$ is a closed vector subspace.
Then $(L\pushfwd A)\infconv B$ is maximally monotone.
\item
\label{l:2iv}
Suppose that $A$ is maximally monotone and that the
cone generated by $\ran L^*+\ran A$ is a closed vector subspace.
Then $L\pushfwd A$ is maximally monotone.
\end{enumerate}
\end{lemma}
\begin{proof}
\ref{l:2i}: This follows easily from \eqref{e:parsumdef} and 
\eqref{e:2011-01-05a}.

\ref{l:2ii}: By \ref{l:2i},
$((L\pushfwd A)\infconv B)^{-1}=L\circ A^{-1}\circ L^*+B^{-1}$.
Since $A^{-1}$ and $B^{-1}$ are monotone and monotonicity is 
preserved under inversion and this type of transformation 
\cite[Proposition~20.10]{Livre1}, the assertion is proved.

\ref{l:2iii}: The operators $A^{-1}$ and $B^{-1}$ are maximally
monotone \cite[Proposition~20.22]{Livre1} and 
$L^*(\ran B)-\ran A=L^*(\dom B^{-1})-\dom A^{-1}$. 
Hence, $L\circ A^{-1}\circ L^*+B^{-1}$ is maximally monotone 
\cite[Section~24]{Botr10} and so is its inverse which, in view 
of \ref{l:2i}, is $(L\pushfwd A)\infconv B$.

\ref{l:2iv}: Set $B=\{0\}^{-1}$ in \ref{l:2iii}.
\end{proof}

\begin{lemma}
\label{l:4}
Let $A\colon\HH\to 2^{\HH}$, let $B\colon\HH\to 2^{\HH}$,
and let $L\in\BL(\HH,\GG)$. 
Then $L\pushfwd(A\infconv B)=(L\pushfwd A)\infconv(L\pushfwd B)$.
\end{lemma}
\begin{proof}
It follows from \eqref{e:parsumdef} and \eqref{e:2011-01-05a} that
\begin{align}
L\pushfwd(A\infconv B)
&=\big(L\circ(A\infconv B)^{-1}\circ L^*\big)^{-1}\nonumber\\
&=\big(L\circ(A^{-1}+B^{-1})\circ L^*\big)^{-1}\nonumber\\
&=\big(L\circ A^{-1}\circ L^*+L\circ B^{-1}\circ L^*\big)^{-1}
\nonumber\\
&=\big((L\pushfwd A)^{-1}+(L\pushfwd B)^{-1}\big)^{-1}
\nonumber\\
&=(L\pushfwd A)\infconv(L\pushfwd B),
\end{align}
which proves the announced identity.
\end{proof}

\begin{lemma}
\label{l:5}
Let $A\colon\HH\to 2^{\HH}$, let $L\in\BL(\HH,\GG)$, and
let $M\in\BL(\GG,\KK)$. Then 
$M\pushfwd(L\pushfwd A)=(M\circ L)\pushfwd A$.
\end{lemma}
\begin{proof}
Indeed, $M\pushfwd(L\pushfwd A)
=(M\circ(L\pushfwd A)^{-1}\circ M^*)^{-1}
=(M\circ L\circ A^{-1}\circ L^* \circ M^*)^{-1}
=(M\circ L)\pushfwd A$.
\end{proof}

Finally, the next lemma draws connections with the infimal 
convolution and postcomposition operations of \eqref{e:J1} 
and \eqref{e:J2}.

\begin{lemma}
\label{l:1}
Let $f\in\Gamma_0(\HH)$, let $g\in\Gamma_0(\GG)$, and let 
$L\in\BL(\HH,\GG)$ be such that
$0\in\sri(L^*(\dom g^*)-\dom f^*)$. Then the following hold.
\begin{enumerate}
\item
\label{l:1i}
$(L\pushfwd f)\infconv g\in\Gamma_0(\GG)$.
\item
\label{l:1ii}
$\partial\big((L\pushfwd f)\infconv g\big)=
(L\pushfwd\partial f)\infconv\partial g$. 
\end{enumerate}
\end{lemma}
\begin{proof}
\ref{l:1i}: Since $0\in L^*(\dom g^*)-\dom f^*$ and, by
the Fenchel-Moreau theorem \cite[Theorem~13.32]{Livre1},
$f^*\in\Gamma_0(\HH)$ and $g^*\in\Gamma_0(\GG)$, we have
$f^*\circ L^*+g^*\in\Gamma_0(\GG)$. Hence
$(f^*\circ L^*+g^*)^*\in\Gamma_0(\GG)$. However, in view of
\cite[Theorem~15.27(i)]{Livre1}, the assumptions also imply that
$(f^*\circ L^*+g^*)^*=(L\pushfwd f)\infconv g$.

\ref{l:1ii}: Let $y$ and $v$ be in $\GG$. Then \ref{l:1i}, 
\cite[Corollary~16.24, Proposition~13.21(i)\&(iv), and
Theorem~16.37(i)]{Livre1} enable us to write
\begin{eqnarray}
v\in\partial\big((L\pushfwd f)\infconv g\big)(y)
&\Leftrightarrow& 
y\in\Big(\partial\big((L\pushfwd f)\infconv g\big)\Big)^{-1}(v)
\nonumber\\
&\Leftrightarrow& 
y\in\partial\big((L\pushfwd f)\infconv g\big)^{*}(v)
\nonumber\\
&\Leftrightarrow& y\in\partial (f^*\circ L^*+g^*)(v)\nonumber\\
&\Leftrightarrow& y\in(L\circ(\partial f^*)\circ L^*+
\partial g^*)(v)\nonumber\\
&\Leftrightarrow& y\in\big(L\circ(\partial f)^{-1}\circ L^*+
(\partial g)^{-1}\big)(v)\nonumber\\
&\Leftrightarrow& v\in\big((L\pushfwd\partial f)
\infconv\partial g\big)y,
\end{eqnarray}
which establishes the announced identity.
\end{proof}

\begin{corollary}
\label{c:1}
Let $f\in\Gamma_0(\HH)$ and let $L\in\BL(\HH,\GG)$ be such that
$0\in\sri(\ran L^*-\dom f^*)$. Then the following hold.
\begin{enumerate}
\item
\label{c:1i}
$L\pushfwd f\in\Gamma_0(\GG)$.
\item
\label{c:1ii}
$\partial (L\pushfwd f)=L\pushfwd\partial f$. 
\end{enumerate}
\end{corollary}
\begin{proof}
Set $g=\iota_{\{0\}}$ in Lemma~\ref{l:1}.
\end{proof}

\subsection{An inclusion problem}

Our main result in Section~\ref{sec:3} will hinge on rewriting 
Problem~\ref{prob:1} as an instance of the following formulation.

\begin{problem}
\label{prob:10}
Let $m$ and $K$ be strictly positive integers, let
$(\HH_i)_{1\leq i\leq m}$ and $(\GG_k)_{1\leq k\leq K}$ be real 
Hilbert spaces, and let $(\mu_i)_{1\leq i\leq m}\in\RP^m$. 
For every $i\in\{1,\ldots,m\}$ 
and $k\in\{1,\ldots,K\}$, let $C_i\colon\HH_i\to\HH_i$ be 
monotone and $\mu_i$-Lipschitzian, let 
$A_i\colon\HH_i\to 2^{\HH_i}$ and $B_k\colon\GG_k\to 2^{\GG_k}$ 
be maximally monotone,  let $z_i\in\HH_i$, 
and let $L_{ki}\in\BL(\HH_i,\GG_k)$. It is assumed that 
\begin{equation}
\label{e:5h8Njiq-11a}
\beta=\sqrt{\lambda}+\underset{1\leq i\leq m}{\text{max}}\mu_i>0,
\quad\text{where}\quad
\lambda\in\left[\underset{\sum_{i=1}^m\|x_i\|^2\leq 1}{\text{sup}}
\sum_{k=1}^K\bigg\|\sum_{i=1}^mL_{ki}x_i\bigg\|^2,\pinf\right[,
\end{equation}
and that the system of coupled inclusions
\begin{multline}
\label{e:IJ834hj8fr-24p}
\text{find}\;\;\overline{x_1}\in\HH_1,\ldots,\overline{x_m}\in\HH_m
\;\;\text{such that}\\
\begin{cases}
z_1&\!\!\!\in A_1\overline{x_1}+\Sum_{k=1}^KL_{k1}^*
\bigg(B_k\bigg(\Sum_{i=1}^mL_{ki}\overline{x_i}\bigg)\bigg)
+C_1\overline{x_1}\\
&\!\!\!\;\vdots\\[-3mm]
z_m&\!\!\!\in A_m\overline{x_m}+\Sum_{k=1}^KL_{km}^*
\bigg(B_k\bigg(\Sum_{i=1}^m
L_{ki}\overline{x_i}\bigg)\bigg)+C_m\overline{x_m}
\end{cases}
\end{multline}
possesses at least one solution. Solve \eqref{e:IJ834hj8fr-24p} 
together with the dual problem
\begin{multline}
\label{e:IJ834hj8fr-24d}
\text{find}\;\;\overline{v_1}\in\GG_1,\ldots,\overline{v_K}
\in\GG_K\;\;\text{such that}\\
\begin{cases}
0&\!\!\!
\in-\Sum_{i=1}^mL_{1i}\big(A_i+C_i\big)^{-1}
\bigg(z_i-\Sum_{k=1}^KL_{ki}^*\overline{v_k}\bigg)
+B_1^{-1}\overline{v_1}\\
&\!\!\!\;\vdots\\[-3mm]
0&\!\!\!\in-\Sum_{i=1}^mL_{Ki}\big(A_i+C_i\big)^{-1}
\bigg(z_i-\Sum_{k=1}^KL_{ki}^*\overline{v_k}\bigg)
+B_K^{-1}\overline{v_K}.
\end{cases}
\end{multline}
\end{problem}

The following result is a special case of 
\cite[Theorem~2.4(iii)]{Siop13}. We use the notation
\eqref{e:jeddah2013-05-23a} to model the possibility of inexactly 
evaluating the operators involved.

\begin{theorem}
\label{t:10}
Consider the setting of Problem~\ref{prob:10}.
Let $x_{1,0}\in\HH_1$, \ldots, $x_{m,0}\in\HH_m$, 
$v_{1,0}\in\GG_1$, \ldots, $v_{K,0}\in\GG_K$, let
$\varepsilon\in\left]0,1/(\beta+1)\right[$, 
let $(\gamma_n)_{n\in\NN}$ be a sequence in 
$[\varepsilon,(1-\varepsilon)/\beta]$, and set
\begin{equation}
\label{e:5h8Njiq-11b}
\begin{array}{l}
\text{For}\;n=0,1,\ldots\\
\left\lfloor
\begin{array}{l}
\text{For}\;i=1,\ldots,m\\
\left\lfloor
\begin{array}{l}
s_{1,i,n}\approx x_{i,n}-\gamma_n\big(C_ix_{i,n}+
\sum_{k=1}^KL_{ki}^*v_{k,n}\big)\\[1mm]
p_{1,i,n}\approx J_{\gamma_n A_i}(s_{1,i,n}+\gamma_nz_i)\\[1mm]
\end{array}
\right.\\[1mm]
\text{For}\;k=1,\ldots,K\\
\left\lfloor
\begin{array}{l}
s_{2,k,n}\approx v_{k,n}+\gamma_n
\sum_{i=1}^mL_{ki}x_{i,n}\\[2mm]
p_{2,k,n}\approx s_{2,k,n}-\gamma_n J_{\gamma_n^{-1}B_k}
(\gamma_n^{-1}s_{2,k,n})\\[2mm]
q_{2,k,n}\approx p_{2,k,n}+\gamma_n\
\sum_{i=1}^mL_{ki}p_{1,i,n}\\
v_{k,n+1}=v_{k,n}-s_{2,k,n}+q_{2,k,n}
\end{array}
\right.\\[1mm]
\text{For}\;i=1,\ldots,m\\
\left\lfloor
\begin{array}{l}
q_{1,i,n}\approx p_{1,i,n}-\gamma_n\big(C_ip_{1,i,n}+
\sum_{k=1}^KL_{ki}^*p_{2,k,n}\big)\\
x_{i,n+1}=x_{i,n}-s_{1,i,n}+q_{1,i,n}.
\end{array}
\right.\\
\end{array}
\right.\\
\end{array}
\end{equation}
Then there exist a solution 
$(\overline{x_1},\ldots,\overline{x_m})$ 
to \eqref{e:IJ834hj8fr-24p} and a solution 
$(\overline{v_1},\ldots,\overline{v_K})$ to \eqref{e:IJ834hj8fr-24d}
such that the following hold.
\begin{enumerate}
\item
\label{t:10iia}
$(\forall i\in\{1,\ldots,m\})$ 
$z_i-\sum_{k=1}^KL_{ki}^*\overline{v_k}\in
A_i\overline{x_i}+C_i\overline{x_i}$.
\item
\label{t:10iib}
$(\forall k\in\{1,\ldots,K\})$
$\sum_{i=1}^mL_{ki}\overline{x_i}\in B_k^{-1}\overline{v_k}$.
\item
\label{t:10iic}
$(\forall i\in\{1,\ldots,m\})$ $x_{i,n}\weakly\overline{x_i}$.
\item
\label{t:10iid}
$(\forall k\in\{1,\ldots,K\})$ $v_{k,n}\weakly\overline{v_k}$.
\item 
\label{t:10iie}
Suppose that $A_1$ or $C_1$ is 
uniformly monotone at $\overline{x_1}$. Then 
$x_{1,n}\to\overline{x_1}$.
\item 
\label{t:10iif}
Suppose that, for some $k\in\{1,\ldots,K\}$, $B_k^{-1}$ is 
uniformly monotone at $\overline{v_k}$. Then
$v_{k,n}\to\overline{v_k}$.
\end{enumerate}
\end{theorem}

\section{Main algorithm}
\label{sec:3}

We start with the following facts.

\begin{proposition}
\label{p:2013-05-16}
Let $\HH$ be a real Hilbert space, let $r$ be a strictly positive
integer, let $z\in\HH$, and let $A\colon\HH\to 2^{\HH}$ and
$C\colon\HH\to\HH$.
For every integer $k\in\{1,\ldots,r\}$, let $\GG_k$ and $\KK_k$
be real Hilbert spaces, let 
$B_k\colon\GG_k\to 2^{\GG_k}$ and $D_k\colon\KK_k\to 2^{\KK_k}$,
and let $L_k\in\BL(\HH,\GG_k)$ and $M_k\in\BL(\HH,\KK_k)$. Set
\begin{equation}
\label{e:2013-05-16e}
\HHH=\bigoplus_{k=1}^r\HH,\quad\GGG=\bigoplus_{k=1}^r\GG_k,
\quad\KKK=\bigoplus_{k=1}^r\KK_k, 
\end{equation}
and
\begin{equation}
\label{e:2013-05-11e}
\begin{cases}
\boldsymbol{A}\colon\HHH\to 2^{\HHH}\colon
(x,y_1,\ldots,y_r)\mapsto
(Ax+Cx-z)\times\{0\}\times\cdots\times\{0\}\\
\boldsymbol{B}\colon\GGG\oplus\KKK\to 2^{\GGG\oplus\KKK}
\colon(s_1,\ldots,s_r,t_1,\ldots,t_r)\mapsto
B_1s_1\times\cdots\times B_rs_r\times D_1t_1\times\cdots
\times D_rt_r\\
\boldsymbol{L}\colon\HHH\to\GGG\oplus\KKK\colon(x,y_1,\ldots,y_r)
\mapsto(L_1x-L_1y_1,\ldots,L_rx-L_ry_r,M_1y_1,\ldots,M_ry_r).
\end{cases}
\end{equation}
Furthermore, suppose that 
\begin{equation}
\label{e:jeddah2013-05-16b}
\big(\exi\overline{\boldsymbol{x}}=(\overline{x},
\overline{y_1},\ldots,\overline{y_r})\in\HHH\big)\quad
\boldsymbol{0}\in\boldsymbol{A}\overline{\boldsymbol{x}}
+\boldsymbol{L}^*\big(\boldsymbol{B}(\boldsymbol{L}
\overline{\boldsymbol{x}})\big).
\end{equation}
Then the following hold for some 
$(\overline{v_1},\ldots,\overline{v_r})\in\GGG$ 
and $(\overline{w_1},\ldots,\overline{w_r})\in\KKK$.
\begin{enumerate}
\item
\label{p:2013-05-16iii}
$z-\sum_{k=1}^rL_{k}^*\overline{v_k}\in
A\overline{x}+C\overline{x}$.
\item
\label{p:2013-05-16iv}
$(\forall k\in\{1,\ldots,r\})$
$L^*_k\overline{v_k}=M_k^*\overline{w_k}$,\;
$L_k\overline{x}-L_k\overline{y}_k\in B_k^{-1}\overline{v_k}$,\;
and\;\,$M_k\overline{y}_k\in D_k^{-1}\overline{w_k}$.
\item
\label{p:2013-05-16i}
$\overline{x}$ solves \eqref{e:2013-04-30p}.
\item
\label{p:2013-05-16ii}
$(\overline{v_1},\ldots,\overline{v_r})$ solves
\eqref{e:2013-04-30d}.
\end{enumerate}
\end{proposition}
\begin{proof}
\ref{p:2013-05-16iii} and \ref{p:2013-05-16iv}:
It follows from \eqref{e:jeddah2013-05-16b} that there exists
$\overline{\boldsymbol{v}}=(\overline{v_1},\ldots,\overline{v_r},
\overline{w_1},\ldots,\overline{w_r})\in\GGG\oplus\KKK$ such that
$-\boldsymbol{L}^*\overline{\boldsymbol{v}}\in\boldsymbol{A}
\overline{\boldsymbol{x}}$ and $\overline{\boldsymbol{v}}\in
\boldsymbol{B}(\boldsymbol{L}\overline{\boldsymbol{x}})$, i.e., 
\begin{equation}
\label{e:jeddah2013-05-16v}
-\boldsymbol{L}^*\overline{\boldsymbol{v}}\in\boldsymbol{A}
\overline{\boldsymbol{x}}
\quad\text{and}\quad 
\boldsymbol{L}\overline{\boldsymbol{x}}\in\boldsymbol{B}^{-1}
\overline{\boldsymbol{v}}.
\end{equation}
Since
\begin{equation}
\label{e:jeddah2013-05-16a}
\boldsymbol{L}^*\colon\GGG\oplus\KKK\to\HHH\colon
(v_1,\ldots,v_r,w_1,\ldots,w_r)
\mapsto\bigg(\sum_{k=1}^rL_k^*v_k,M_1^*w_1-L_1^*v_1,\ldots,
M_r^*w_r-L_r^*v_r\bigg),
\end{equation}
it follows from \eqref{e:2013-05-11e} that 
\eqref{e:jeddah2013-05-16v} can be rewritten as
\begin{equation}
\label{e:jeddah2013-05-16g}
z-\sum_{k=1}^rL_{k}^*\overline{v_k}\in A\overline{x}+C\overline{x}
\quad\text{and}\quad
(\forall k\in\{1,\ldots,r\})\quad
\begin{cases}
L^*_k\overline{v_k}=M_k^*\overline{w_k}\\
L_k\overline{x}-L_k\overline{y}_k\in B_k^{-1}\overline{v_k}\\
M_k\overline{y}_k\in D_k^{-1}\overline{w_k}.
\end{cases}
\end{equation}

\ref{p:2013-05-16i}:
For every $k\in\{1,\ldots,r\}$,
\begin{align}
\text{\ref{p:2013-05-16iv}}
&\Rightarrow
\begin{cases}
L^*_k\overline{v_k}=M_k^*\overline{w_k}\\
\overline{v_k}\in B_k(L_k\overline{x}-L_k\overline{y}_k)\\
\overline{w_k}\in D_k(M_k\overline{y}_k)
\end{cases}
\label{e:2013-05-05R}\\
&\Rightarrow
\begin{cases}
L^*_k\overline{v_k}=M_k^*\overline{w_k}\\
L_k^*\overline{v_k}\in L_k^*
\big(B_k(L_k\overline{x}-L_k\overline{y}_k)\big)\\
M_k^*\overline{w_k}\in M_k^*\big(D_k(M_k\overline{y}_k)\big)
\end{cases}
\nonumber\\
&\Leftrightarrow
\begin{cases}
L^*_k\overline{v_k}=M_k^*\overline{w_k}\\
\overline{x}-\overline{y}_k\in
(L_k^*\circ B_k\circ L_k)^{-1}(L_k^*\overline{v_k})\\
\overline{y}_k\in
(M_k^*\circ D_k\circ M_k)^{-1}(M_k^*\overline{w_k})
\end{cases}
\label{e:2013-05-05S}\\
&\Rightarrow
\overline{x}\in
(L_k^*\circ B_k\circ L_k)^{-1}(L_k^*\overline{v_k})+
(M_k^*\circ D_k\circ M_k)^{-1}(L_k^*\overline{v_k})
\nonumber\\
&\Leftrightarrow
L_k^*\overline{v_k}\in
\big((L_k^*\circ B_k\circ L_k)\infconv 
(M_k^*\circ D_k\circ M_k)\big)(\overline{x}).
\label{e:jeddah2013-05-16m}
\end{align}
Hence, 
\begin{equation}
\label{e:2013-05-05D}
\sum_{k=1}^rL_k^*\overline{v_k}\in\sum_{k=1}^r
\big((L_k^*\circ B_k\circ L_k)\infconv 
(M_k^*\circ D_k\circ M_k)\big)(\overline{x}).
\end{equation}
Adding this inclusion to that of \ref{p:2013-05-16iii} shows that
$\overline{x}$ solves \eqref{e:2013-04-30p}.

\ref{p:2013-05-16ii}:
It follows from \ref{p:2013-05-16iii} that
\begin{equation}
\label{e:2013-05-06A}
(\forall k\in\{1,\ldots,r\})\quad {-L_k}\overline{x}\in
-L_k\bigg((A+C)^{-1}\bigg(z-\sum_{l=1}^rL_{l}^*
\overline{v_l}\bigg)\bigg).
\end{equation}
On the other hand, \ref{p:2013-05-16iv} yields
\begin{align}
\label{e:2013-05-06B}
(\forall k\in\{1,\ldots,r\})\quad 
L_k\overline{x}-L_k\overline{y_k}\in 
B_k^{-1}\overline{v_k},
\end{align}
while \eqref{e:2013-05-05S} yields
\begin{align}
\label{e:2013-05-06C}
(\forall k\in\{1,\ldots,r\})\quad 
L_k\overline{y_k}
&\in L_k\big((M_k^*\circ D_k\circ M_k)^{-1}
(M_k^*\overline{w_k})\big)\nonumber\\
&=L_k\big((M_k^*\circ D_k\circ M_k)^{-1}(L_k^*\overline{v_k})\big)
\nonumber\\
&=L_k\big((M_k^*\pushfwd D_k^{-1})(L_k^*\overline{v_k})\big).
\end{align}
Upon adding \eqref{e:2013-05-06A}, \eqref{e:2013-05-06B}, and 
\eqref{e:2013-05-06C}, we obtain
\begin{multline}
\label{e:2013-05-06D}
(\forall k\in\{1,\ldots,r\})\quad 
0\in-L_k\bigg((A+C)^{-1}\bigg(z-\sum_{l=1}^rL_{l}^*
\overline{v_l}\bigg)\bigg)\\
+B_k^{-1}\overline{v_k}
+L_k\big((M_k^*\pushfwd D_k^{-1})(L_k^*\overline{v_k})\big),
\end{multline}
which proves that
$(\overline{v_1},\ldots,\overline{v_r})$ solves
\eqref{e:2013-04-30d}.
\end{proof}

We are now in a position to present our main result.

\begin{theorem}
\label{t:1}
Consider the setting of Problem~\ref{prob:1}. Let $x_{0}\in\HH$, 
$y_{1,0}\in\HH$, \ldots, $y_{r,0}\in\HH$, $v_{1,0}\in\GG_1$, 
\ldots, $v_{r,0}\in\GG_r$, $w_{1,0}\in\KK_1$, \ldots, 
$w_{r,0}\in\KK_r$, 
let $\varepsilon\in\left]0,1/(\beta+1)\right[$,
let $(\gamma_n)_{n\in\NN}$ be a sequence in 
$[\varepsilon,(1-\varepsilon)/\beta]$, and set
\begin{equation}
\label{e:2013-05-06x}
\begin{array}{l}
\text{For}\;n=0,1,\ldots\\
\left\lfloor
\begin{array}{l}
s_{1,1,n}\approx x_{n}-\gamma_n(Cx_{n}+
\sum_{k=1}^rL_{k}^*v_{k,n})\\[1mm]
p_{1,1,n}\approx J_{\gamma_n A}(s_{1,1,n}+\gamma_nz)\\[2mm]
\text{For}\;k=1,\ldots,r\\
\left\lfloor
\begin{array}{l}
p_{1,k+1,n}\approx y_{k,n}+\gamma_n(L_{k}^*v_{k,n}-
M_{k}^*w_{k,n})\\[1mm]
s_{2,k,n}\approx v_{k,n}+\gamma_nL_k(x_n-y_{k,n})\\[2mm]
p_{2,k,n}\approx s_{2,k,n}-\gamma_n J_{\gamma_n^{-1}B_k}
(\gamma_n^{-1}s_{2,k,n})\\[2mm]
q_{2,k,n}\approx p_{2,k,n}+
\gamma_nL_k(p_{1,1,n}-p_{1,k+1,n})\\[2mm]
v_{k,n+1}=v_{k,n}-s_{2,k,n}+q_{2,k,n}\\[2mm]
s_{2,k+r,n}\approx w_{k,n}+\gamma_nM_ky_{k,n}\\[2mm]
p_{2,k+r,n}\approx s_{2,k+r,n}-\gamma_n\big(J_{\gamma_n^{-1}D_k}
(\gamma_n^{-1}s_{2,k+r,n})\big)\\[2mm]
q_{1,k+1,n}\approx p_{1,k+1,n}+\gamma_n
(L_k^*p_{2,k,n}-M_k^*p_{2,k+r,n})\\[2mm]
q_{2,k+r,n}\approx p_{2,k+r,n}+\gamma_nM_{k}p_{1,k+1,n}\\[1mm]
w_{k,n+1}=w_{k,n}-s_{2,k+r,n}+q_{2,k+r,n}\\
\end{array}
\right.\\[2mm]
q_{1,1,n}\approx p_{1,1,n}-\gamma_n(Cp_{1,1,n}+
\sum_{k=1}^rL_k^*p_{2,k,n})\\
x_{n+1}=x_{n}-s_{1,1,n}+q_{1,1,n}\\[1mm]
\text{For}\;k=1,\ldots,r\\
\left\lfloor
\begin{array}{l}
y_{k,n+1}=y_{k,n}-p_{1,k+1,n}+q_{1,k+1,n}.
\end{array}
\right.\\
\end{array}
\right.\\
\end{array}
\end{equation}
Then the following hold for some solution $\overline{x}$ to 
\eqref{e:2013-04-30p} and some solution 
$(\overline{v_1},\ldots,\overline{v_r})$ to \eqref{e:2013-04-30d}.
\begin{enumerate}
\item
\label{t:1i}
$x_n\weakly\overline{x}~$ and 
$~(\forall k\in\{1,\ldots,r\})$ $v_{k,n}\weakly\overline{v_k}$.
\item
\label{t:1ii}
Suppose that $A$ or $C$ is uniformly monotone at 
$\overline{x}$. Then $x_{n}\to\overline{x}$.
\item
\label{t:1iii}
Suppose that, for some $k\in\{1,\ldots,r\}$, $B^{-1}_k$ 
is uniformly monotone at $\overline{v_k}$. Then
$v_{k,n}\to\overline{v_k}$.
\end{enumerate}
\end{theorem}
\begin{proof}
We introduce the auxiliary problem
\begin{multline}
\label{e:2013-05-04a}
\text{find}\;\;\overline{x}\in\HH,\,\overline{y_1}\in\HH,\ldots,\,
\overline{y_r}\in\HH\;\;\text{such that}\\
\begin{cases}
z&\in A\overline{x}+\sum_{k=1}^rL_k^*\big(B_k(L_k\overline{x}-
L_k\overline{y_k})\big)+C\overline{x}\\
0&\in -L_1^*\big(B_1(L_1\overline{x}-L_1\overline{y_1})\big)
+M_1^*\big(D_1(M_1\overline{y_1})\big)\\
&\;\vdots\\
0&\in -L_r^*\big(B_r(L_r\overline{x}-L_r\overline{y_r})\big)
+M_r^*\big(D_r(M_r\overline{y_r})\big).
\end{cases}
\end{multline}
Let $x\in\HH$. Then 
\begin{align}
\label{e:2013-05-02x}
x~\text{solves \eqref{e:2013-04-30p}}
&\Leftrightarrow
z\in Ax+
\sum_{k=1}^r\big((L_k^*\circ B_k\circ L_k)\infconv
(M_k^*\circ D_k\circ M_k)\big)x+Cx\nonumber\\
&\Leftrightarrow 
\big(\exi(u_k)_{1\leq k\leq r}\in\HH^r\big)\quad
\begin{cases}
z&\in A{x}+\sum_{k=1}^ru_k+C{x}\\
u_1&\in\big((L_1^*\circ B_1\circ L_1)\infconv
(M_1^*\circ D_1\circ M_1)\big){x}\\
&\;\vdots\\
u_r&\in\big((L_r^*\circ B_r\circ L_r)\infconv
(M_r^*\circ D_r\circ M_r)\big){x}
\end{cases}
\nonumber\\
&\Leftrightarrow 
\big(\exi(u_k)_{1\leq k\leq r}\in\HH^r\big)\quad
\begin{cases}
z&\in A{x}+\sum_{k=1}^ru_k+C{x}\\
x&\in (L_1^*\circ B_1\circ L_1)^{-1}u_1+
(M_1^*\circ D_1\circ M_1)^{-1}u_1\\
&\;\vdots\\
x&\in (L_r^*\circ B_r\circ L_r)^{-1}u_r+
(M_r^*\circ D_r\circ M_r)^{-1}u_r
\end{cases}
\nonumber\\
&\Leftrightarrow 
\big(\exi(u_k)_{1\leq k\leq r}\in\HH^r\big)
\big(\exi(y_k)_{1\leq k\leq r}\in\HH^r\big)\nonumber\\
&\quad\;\;
\qquad\begin{cases}
z&\in A{x}+\sum_{k=1}^ru_k+C{x}\\
x-y_1&\in(L_1^*\circ B_1\circ L_1)^{-1}u_1\\
y_1&\in(M_1^*\circ D_1\circ M_1)^{-1}u_1\\
&\;\vdots\\
x-y_r&\in(L_r^*\circ B_r\circ L_r)^{-1}u_r\\
y_r&\in(M_r^*\circ D_r\circ M_r)^{-1}u_r
\end{cases}
\nonumber\\
&\Leftrightarrow 
\big(\exi(u_k)_{1\leq k\leq r}\in\HH^r\big)
\big(\exi(y_k)_{1\leq k\leq r}\in\HH^r\big)\nonumber\\
&\quad\;\;
\qquad\begin{cases}
z&\in A{x}+\sum_{k=1}^ru_k+C{x}\\
u_1&\in(L_1^*\circ B_1\circ L_1)(x-y_1)\\
u_1&\in(M_1^*\circ D_1\circ M_1)y_1\\
&\;\vdots\\
u_r&\in(L_r^*\circ B_r\circ L_r)(x-y_r)\\
u_r&\in(M_r^*\circ D_r\circ M_r)y_r
\end{cases}
\nonumber\\
&\Rightarrow 
\big(\exi(y_k)_{1\leq k\leq r}\in\HH^r\big)\nonumber\\
&\quad\;\;
\qquad\begin{cases}
z&\in A{x}+\sum_{k=1}^rL_k^*\big(B_k(L_kx-L_ky_k)\big)+C{x}\\
0&\in -L_1^*\big(B_1(L_1x-L_1y_1)\big)+M_1^*\big(D_1(M_1y_1)\big)\\
&\;\vdots\\
0&\in -L_r^*\big(B_r(L_rx-L_ry_r)\big)+M_r^*\big(D_r(M_ry_r)\big).
\end{cases}
\end{align}
Hence since, by assumption, \eqref{e:2013-04-30p} has at least one 
solution, 
\begin{equation}
\label{e:2013-05-07c}
\text{\eqref{e:2013-05-04a} has at least one solution.}
\end{equation}
Next, we set
\begin{equation}
\label{e:bataclan2013-03-20a}
\begin{cases}
m=r+1\\
K=2r\\
\HH_1=\HH\\
A_1=A\\
C_1=C\\
\mu_1=\mu\\
z_1=z
\end{cases}
\quad\text{and}\quad
(\forall k\in\{1,\ldots,r\})\quad
\begin{cases}
\HH_{k+1}=\HH\\
A_{k+1}=0\\
B_{k+r}=D_k\\
C_{k+1}=0\\
\mu_{k+1}=0\\
z_{k+1}=0.
\end{cases}
\end{equation}
We also define
\begin{multline}
\label{e:guadeloupe2012-10-29e}
(\forall k\in\{1,\ldots,r\})\quad 
\GG_{k+r}=\KK_k
\quad\text{and}\\
(\forall i\in\{1,\ldots,m\})\quad 
L_{ki}=
\begin{cases}
L_k,&\text{if}\;\;1\leq k\leq r\;\;\text{and}\;\;i=1;\\
-L_k,&\text{if}\;\;1\leq k\leq r\;\;\text{and}\;\;i=k+1;\\
M_{k-r},&\text{if}\;\;r+1\leq k\leq 2r\;\;\text{and}\;\;i=k-r+1;\\
0,&\text{otherwise.}
\end{cases}
\end{multline}
We observe that in this setting
\begin{equation}
\label{e:2013-05-07a}
\text{\eqref{e:2013-05-04a} is a special case of 
\eqref{e:IJ834hj8fr-24p}.}
\end{equation}
Moreover, if we set $\lambda=\sum_{k=1}^r\|L_k\|^2+
\max_{1\leq k\leq r}(\|L_k\|^2+\|M_k\|^2)$, we deduce from the 
Cauchy-Schwarz inequality in $\RR^2$ that, for every 
$(x_i)_{1\leq i\leq m}=
(x,y_1,\ldots,y_{r})\in\bigoplus_{i=1}^m\HH$, 
\begin{align}
\label{e:roma12}
\sum_{k=1}^K\bigg\|\sum_{i=1}^mL_{ki}x_i\bigg\|^2
&=\|(L_1x-L_1y_1,\ldots,L_rx-L_ry_r,M_1y_1,\ldots,M_ry_r)\|^2
\nonumber\\
&\leq\big(\|(L_1x,\ldots,L_rx)\|+\|(L_1y_1,\ldots,L_ry_r,
M_1y_1,\ldots,M_ry_r)\|\big)^2
\nonumber\\
&=\left(\sqrt{\sum_{k=1}^r\|L_kx\|^2}
+\sqrt{\sum_{k=1}^r\big(\|L_ky_k\|^2+\|M_ky_k\|^2\big)}\,\right)^2
\nonumber\\
&\leq\left(\sqrt{\sum_{k=1}^r\|L_k\|^2}\:\|x\|
+\sqrt{\sum_{k=1}^r(\|L_k\|^2+\|M_k\|^2)\,\|y_k\|^2}\,\right)^2
\nonumber\\
&\leq\left(\sqrt{\sum_{k=1}^r\|L_k\|^2}\:\|x\|
+\max_{1\leq k\leq r}\sqrt{\|L_k\|^2+
\|M_k\|^2}\:\|(y_1,\ldots,y_r)\|\,\right)^2
\nonumber\\
&\leq\left(\sum_{k=1}^r\|L_k\|^2
+\max_{1\leq k\leq r}\big(\|L_k\|^2+\|M_k\|^2\big)\right)
\big(\|x\|^2+\|(y_1,\ldots,y_r)\|^2\big)
\nonumber\\
&=\lambda\sum_{i=1}^m\|x_i\|^2.
\end{align}
Thus, 
\begin{equation}
\label{e:2013-05-07b}
\text{\eqref{e:2013-04-30b} is a special case of 
\eqref{e:5h8Njiq-11a}.}
\end{equation}
Now, let us define
\begin{equation}
\label{e:2013-05-05a}
(\forall n\in\NN)\quad x_{1,n}=x_n\quad\text{and}\quad
(\forall k\in\{1,\ldots,r\})\quad 
\begin{cases}
x_{k+1,n}=y_{k,n}\\
v_{k+r,n}=w_{k,n}.
\end{cases}
\end{equation}
Then it follows from \eqref{e:bataclan2013-03-20a} that
\begin{equation}
\label{e:2013-05-07d}
\text{\eqref{e:2013-05-06x} is a special case of 
\eqref{e:5h8Njiq-11b}.}
\end{equation}
Altogether, Theorem~\ref{t:10}\ref{t:10iia}--\ref{t:10iid} asserts 
that there exist a solution 
$\overline{\boldsymbol{x}}
=(\overline{x_1},\ldots,\overline{x_{m}})= 
(\overline{x},\overline{y_1},\ldots,\overline{y_{r}})$ 
to \eqref{e:IJ834hj8fr-24p} and a solution 
$(\overline{v_1},\ldots,\overline{v_K})=
(\overline{v_1},\ldots,\overline{v_r},
\overline{w_1},\ldots,\overline{w_r})$ to 
\eqref{e:IJ834hj8fr-24d} which satisfy
\begin{equation}
\label{e:2013-05-05m}
x_n\weakly\overline{x}\quad\text{and}\quad
(\forall k\in\{1,\ldots,r\})\quad 
\begin{cases}
v_{k,n}\weakly\overline{v_k}\\
w_{k,n}\weakly\overline{w_k},
\end{cases}
\end{equation}
together with the inclusions
\begin{multline}
\label{e:2013-05-05b}
z-\sum_{k=1}^rL_{k}^*\overline{v_k}\in A\overline{x}+C\overline{x}
\quad\text{and}\quad
(\forall k\in\{1,\ldots,r\})\quad
\begin{cases}
L^*_k\overline{v_k}=M_k^*\overline{w_k}\\
L_k\overline{x}-L_k\overline{y}_k\in B_k^{-1}\overline{v_k}\\
M_k\overline{y}_k\in D_k^{-1}\overline{w_k}.
\end{cases}
\end{multline}
Using the notation \eqref{e:2013-05-11e}, we can rewrite 
\eqref{e:2013-05-05b} as 
\begin{equation}
\label{e:attento}
\boldsymbol{0}\in\boldsymbol{A}\overline{\boldsymbol{x}}
+\boldsymbol{L}^*\big(\boldsymbol{B}(\boldsymbol{L}
\overline{\boldsymbol{x}})\big).
\end{equation}
In turn, it follows from Proposition~\ref{p:2013-05-16}%
\ref{p:2013-05-16i}--\ref{p:2013-05-16ii} that
\begin{equation}
\label{e:2013-05-05e}
\overline{x}~\text{solves \eqref{e:2013-04-30p} and}~
(\overline{v_1},\ldots,\overline{v_r})~
\text{solves \eqref{e:2013-04-30d}}.
\end{equation}
This and \eqref{e:2013-05-05m} prove \ref{t:1i}. Finally, 
\ref{t:1ii}--\ref{t:1iii} follow from 
\eqref{e:bataclan2013-03-20a} and
Theorem~\ref{t:10}\ref{t:10iie}--\ref{t:10iif}.
\end{proof}

\begin{remark}
In the spirit of the splitting method of \cite{Siop13,Svva12}, 
the algorithm described in \eqref{e:2013-05-06x} achieves full 
decomposition in that every operator is used individually 
at each iteration.
\end{remark}

\section{Application to convex minimization}
\label{sec:4}

In this section we consider a structured minimization problem 
of the following format.

\begin{problem}
\label{prob:2}
Let $\HH$ be a real Hilbert space, let $r$ be a strictly positive
integer, let $z\in\HH$, let $f\in\Gamma_0(\HH)$, and let 
$\ell\colon\HH\to\RR$ be a differentiable convex function such
that $\nabla\ell$ is $\mu$-Lipschitzian for some $\mu\in\RP$. 
For every integer $k\in\{1,\ldots,r\}$, let $\GG_k$ and $\KK_k$
be real Hilbert spaces, let $g_k\in\Gamma_0(\GG_k)$ and
$h_k\in\Gamma_0(\KK_k)$, and let $L_k\in\BL(\HH,\GG_k)$ and
$M_k\in\BL(\HH,\KK_k)$. It is assumed that
\begin{equation}
\label{e:2013-05-08b}
\beta=\mu+\sqrt{\sum_{k=1}^r\|L_k\|^2+\max_{1\leq k\leq r}
\big(\|L_k\|^2+\|M_k\|^2\big)}\:>0,
\end{equation}
that
\begin{equation}
\label{e:2013-05-12q}
(\forall k\in\{1,\ldots,r\})\quad
0\in\sri\big(\dom(g_k\circ L_k)^*-M_k^*(\dom h^*_k)\big),
\end{equation}
that
\begin{equation}
\label{e:2013-05-17q}
(\forall k\in\{1,\ldots,r\})\quad
0\in\sri\big(\ran M_k-\dom h_k\big),
\end{equation}
and that 
\begin{equation}
\label{e:R-stay}
z\in\ran\bigg(
\partial f+\sum_{k=1}^r\big((L_k^*\circ(\partial g_k)\circ L_k)
\infconv (M_k^*\circ(\partial h_k)\circ M_k)\big)+
\nabla\ell\bigg).
\end{equation}
Solve the primal problem 
\begin{equation}
\label{e:2013-05-08p}
\minimize{x\in\HH}{f(x)+\sum_{k=1}^r\big( (g_k\circ L_k) 
\infconv(h_k\circ M_k)\big)(x)+\ell(x)-\scal{x}{z}},
\end{equation}
together with the dual problem
\begin{equation}
\label{e:2013-05-08d}
\minimize{v_1\in\GG_1,\ldots,\,v_r\in\GG_r}{(f^*\infconv\ell^*)
\bigg(z-\sum_{k=1}^rL_{k}^*v_k\bigg)+\sum_{k=1}^r
\big(g^*_k(v_k)+(M^*_k\pushfwd h^*_k)(L^*_kv_k)\big)}.
\end{equation}
\end{problem}

Special cases when \eqref{e:2013-05-12q} and \eqref{e:2013-05-17q} 
are satisfied can be derived from \cite[Proposition~15.24]{Livre1}.
The next proposition describes scenarios in which 
\eqref{e:R-stay} holds.

\begin{proposition}
\label{p:2013-05-17b}
Consider the same setting as in Problem~\ref{prob:2} with the
exception that assumption \eqref{e:R-stay} is not made and is
replaced by the assumptions that 
\begin{multline}
\label{e:cq}
\boldsymbol{E}=\Big\{\big(L_1(x-y_1)-s_1,\ldots,L_r(x-y_r)-s_r,
M_1y_1-t_1,\ldots,M_ry_r-t_r\big)~\Big|~x\in\dom f,\\
y_1\in\HH,\ldots, y_r\in\HH,
s_1\in\dom g_1,\ldots,s_r\in\dom g_r,
t_1\in\dom h_1,\ldots,t_r\in\dom h_r\Big\}\neq\emp
\end{multline}
and that \eqref{e:2013-05-08p} has a solution. 
Then \eqref{e:R-stay} is satisfied in each of the following
cases.
\begin{enumerate}
\item
\label{p:2013-05-17bi}
$\boldsymbol{0}\in\sri\boldsymbol{E}$.
\item
\label{p:2013-05-17bii}
$\boldsymbol{E}$ is a closed vector subspace.
\item
\label{p:2013-05-17biii}
$f$ is real-valued and, for every $k\in\{1,\ldots,r\}$, the 
operators $L_k$ and $M_k$ are surjective. 
\item
\label{p:2013-05-17biv}
For every $k\in\{1,\ldots,r\}$, $g_k$ and $h_k$ are real-valued.
\item
\label{p:2013-05-17bv}
$\HH$, $(\GG_k)_{1\leq k\leq r}$, and $(\KK_k)_{1\leq k\leq r}$ 
are finite-dimensional, and 
\begin{equation}
\label{e:cq2}
(\exi x\in\reli\dom f)(\forall k\in\{1,\ldots,K\})(\exi y_k\in\HH)
\quad 
\begin{cases}
L_k(x-y_k)\in\reli\dom g_k\\
M_ky_k\in\reli\dom h_k.
\end{cases}
\end{equation}
\end{enumerate}
\end{proposition}
\begin{proof}
Let us define $\HHH$, $\GGG$, and $\KKK$ as in 
\eqref{e:2013-05-16e}, $\boldsymbol{L}$ as in 
\eqref{e:2013-05-11e}, and let us set
\begin{equation}
\label{e:2013-05-11E}
\begin{cases}
\boldsymbol{f}\colon\HHH\to\RX\colon
\boldsymbol{x}=(x,y_1,\ldots,y_r)\mapsto
f(x)+\ell(x)-\scal{x}{z}\\
\boldsymbol{g}\colon\GGG\oplus\KKK\to\RX\colon\boldsymbol{s}
=(s_1,\ldots,s_r,t_1,\ldots,t_r)\mapsto
\sum_{k=1}^r\big(g_k(s_k)+h_k(t_k)\big).
\end{cases}
\end{equation}
Then we can rewrite \eqref{e:cq} as
\begin{equation}
\label{e:2013-05-17m}
\boldsymbol{E}=\boldsymbol{L}
(\dom\boldsymbol{f})-\dom\boldsymbol{g}.
\end{equation}

\ref{p:2013-05-17bi}:
Since $\boldsymbol{E}\neq\emp$, the functions
$(g_k\circ L_k)_{1\leq k\leq r}$ and
$(h_k\circ M_k)_{1\leq k\leq r}$ are proper
and therefore in $\Gamma_0(\HH)$.
In turn, the Fenchel-Moreau theorem 
\cite[Theorem~13.32]{Livre1} asserts that the functions
$((g_k\circ L_k)^*)_{1\leq k\leq r}$ and
$((h_k\circ M_k)^*)_{1\leq k\leq r}$ are in $\Gamma_0(\HH)$.
On the other hand, since \eqref{e:2013-05-17q} and
\cite[Corollary~15.28(i)]{Livre1} imply that
\begin{equation}
\label{e:2013-05-18a}
(\forall k\in\{1,\ldots,r\})\quad
(h_k\circ M_k)^*=M_k^*\pushfwd h_k^*,
\end{equation}
\eqref{e:2013-05-12q} and \cite[Proposition~12.34(i)]{Livre1} 
yield
\begin{align}
\label{e:2013-05-18q}
(\forall k\in\{1,\ldots,r\})\quad
0
&\in\sri\big(\dom(g_k\circ L_k)^*-M_k^*(\dom h^*_k)\big)\nonumber\\
&=\sri\big(\dom(g_k\circ L_k)^*-
\dom(M_k^*\pushfwd h_k^*)\big)\nonumber\\
&=\sri\big(\dom(g_k\circ L_k)^*-
\dom(h_k\circ M_k)^*\big).
\end{align}
Hence, we derive from \cite[Proposition~15.7]{Livre1} that
\begin{multline}
\label{e:2013-05-11a}
(\forall k\in\{1,\ldots,r\})(\forall x\in\HH)(\exi y_k\in\HH)\\
\big((g_k\circ L_k)\infconv(h_k\circ M_k)\big)(x)=
g_k(L_kx-L_ky_k)+h_k(M_ky_k),
\end{multline}
which allows us to rewrite \eqref{e:2013-05-08p} as a minimization
problem on $\HHH$, namely
\begin{equation}
\label{e:2013-05-11p}
\minimize{x\in\HH,y_1\in\HH,\ldots,y_r\in\HH}
{f(x)+\ell(x)-\scal{x}{z}+\sum_{k=1}^r\big(g_k(L_kx-L_ky_k)
+h_k(M_ky_k)\big)}
\end{equation}
or, equivalently, 
\begin{equation}
\label{e:2013-05-11p'}
\minimize{\boldsymbol{x}\in\HHH}
{\boldsymbol{f}(\boldsymbol{x})+\boldsymbol{g}
(\boldsymbol{L}\boldsymbol{x})}.
\end{equation}
It follows from \eqref{e:2013-05-17m} that
$\boldsymbol{0}\in\sri\big(\boldsymbol{L}(\dom\boldsymbol{f})-
\dom\boldsymbol{g}\big)$
and therefore from \cite[Theorem~16.37(i)]{Livre1}, that
\begin{equation}
\label{e:2013-05-11x}
\partial(\boldsymbol{f}+\boldsymbol{g}\circ
\boldsymbol{L})=\partial\boldsymbol{f}+
\boldsymbol{L}^*\circ(\partial \boldsymbol{g})\circ \boldsymbol{L}.
\end{equation}
Since, by assumption, \eqref{e:2013-05-08p} has a solution, so does
\eqref{e:2013-05-11p'}. By Fermat's rule 
\cite[Theorem~16.2]{Livre1}, this means that 
$\boldsymbol{0}\in\ran\partial(\boldsymbol{f}+\boldsymbol{g}\circ
\boldsymbol{L})$. Thus \eqref{e:2013-05-11x} yields
\begin{equation}
\label{e:2013-05-11f}
\boldsymbol{0}
\in\ran\big(\partial\boldsymbol{f}+
\boldsymbol{L}^*\circ(\partial\boldsymbol{g})\circ
\boldsymbol{L}\big).
\end{equation}
Let us introduce the operators
\begin{equation}
\label{e:2013-05-09a-}
A=\partial f,\;C=\nabla\ell,\quad\text{and}\quad
(\forall k\in\{1,\ldots,r\})\quad
\begin{cases}
B_k=\partial g_k\\
D_k=\partial h_k. 
\end{cases}
\end{equation}
We derive from \cite[Proposition~17.10]{Livre1} that $C$ is 
monotone and from \cite[Theorem~20.40]{Livre1}
that the operators $A$, $(B_k)_{1\leq k\leq r}$, and 
$(D_k)_{1\leq k\leq r}$ are maximally monotone. 
Next, let us define $\boldsymbol{A}$ and $\boldsymbol{B}$ as 
in \eqref{e:2013-05-11e}. Then it follows from 
\eqref{e:2013-05-11f} and \cite[Proposition~16.8]{Livre1} 
that \eqref{e:jeddah2013-05-16b} holds.
In turn, Proposition~\ref{p:2013-05-16}\ref{p:2013-05-16i} asserts
that \eqref{e:R-stay} is satisfied.

\ref{p:2013-05-17bii}$\Rightarrow$\ref{p:2013-05-17bi}:
This follows from 
\cite[Proposition~6.19(i)]{Livre1}. 

\ref{p:2013-05-17biii}$\Rightarrow$\ref{p:2013-05-17bi} and
\ref{p:2013-05-17biv}$\Rightarrow$\ref{p:2013-05-17bi}:
In both cases $\boldsymbol{E}=\GGG\oplus\KKK$.

\ref{p:2013-05-17bv}$\Rightarrow$\ref{p:2013-05-17bi}:
Since $\HHH$, $\GGG$, and $\KKK$ are finite-dimensional,
\eqref{e:2013-05-17m} and
\cite[Corollary~6.15]{Livre1} imply that
\begin{align}
\text{\eqref{e:cq2}}
&\Leftrightarrow 
(\exi\boldsymbol{x}\in\reli\dom \boldsymbol{f})\quad
\boldsymbol{L}\boldsymbol{x}\in\reli\dom \boldsymbol{g}\nonumber\\
&\Leftrightarrow 
\boldsymbol{0}\in\big(\boldsymbol{L}
(\reli\dom\boldsymbol{f})-\reli\dom\boldsymbol{g}\big)\nonumber\\
&\Leftrightarrow 
\boldsymbol{0}\in\reli\big(\boldsymbol{L}
(\dom\boldsymbol{f})-\dom\boldsymbol{g}\big)\nonumber\\
&\Leftrightarrow 
\boldsymbol{0}\in\reli\boldsymbol{E}\nonumber\\
&\Leftrightarrow 
\boldsymbol{0}\in\sri\boldsymbol{E},
\end{align}
which completes the proof.
\end{proof}

Next, we propose our algorithm for solving Problem~\ref{prob:2}.

\begin{theorem}
\label{t:2}
Consider the setting of Problem~\ref{prob:2}. Let $x_{0}\in\HH$, 
$y_{1,0}\in\HH$, \ldots, $y_{r,0}\in\HH$, $v_{1,0}\in\GG_1$, 
\ldots, $v_{r,0}\in\GG_r$, $w_{1,0}\in\KK_1$, \ldots, 
$w_{r,0}\in\KK_r$, 
let $\varepsilon\in\left]0,1/(\beta+1)\right[$,
let $(\gamma_n)_{n\in\NN}$ be a sequence in 
$[\varepsilon,(1-\varepsilon)/\beta]$, and set
\begin{equation}
\label{e:2013-05-08x}
\begin{array}{l}
\text{For}\;n=0,1,\ldots\\
\left\lfloor
\begin{array}{l}
s_{1,1,n}\approx x_{n}-\gamma_n(\nabla\ell(x_{n})+
\sum_{k=1}^rL_{k}^*v_{k,n})\\[1mm]
p_{1,1,n}\approx \prox_{\gamma_n f}(s_{1,1,n}+\gamma_nz)\\[2mm]
\text{For}\;k=1,\ldots,r\\
\left\lfloor
\begin{array}{l}
p_{1,k+1,n}\approx y_{k,n}+\gamma_n(L_{k}^*v_{k,n}-
M_{k}^*w_{k,n})\\[1mm]
s_{2,k,n}\approx v_{k,n}+\gamma_nL_k(x_n-y_{k,n})\\[2mm]
p_{2,k,n}\approx s_{2,k,n}-\gamma_n \prox_{\gamma_n^{-1}g_k}
(\gamma_n^{-1}s_{2,k,n})\\[2mm]
q_{2,k,n}\approx p_{2,k,n}+
\gamma_nL_k(p_{1,1,n}-p_{1,k+1,n})\\[2mm]
v_{k,n+1}=v_{k,n}-s_{2,k,n}+q_{2,k,n}\\[2mm]
s_{2,k+r,n}\approx w_{k,n}+\gamma_nM_ky_{k,n}\\[2mm]
p_{2,k+r,n}\approx s_{2,k+r,n}-\gamma_n\big(\prox_{\gamma_n^{-1}h_k}
(\gamma_n^{-1}s_{2,k+r,n})\big)\\[2mm]
q_{1,k+1,n}\approx p_{1,k+1,n}+\gamma_n
(L_k^*p_{2,k,n}-M_k^*p_{2,k+r,n})\\[2mm]
q_{2,k+r,n}\approx p_{2,k+r,n}+\gamma_nM_{k}p_{1,k+1,n}\\[1mm]
w_{k,n+1}=w_{k,n}-s_{2,k+r,n}+q_{2,k+r,n}\\
\end{array}
\right.\\[2mm]
q_{1,1,n}\approx p_{1,1,n}-\gamma_n(\nabla\ell(p_{1,1,n})+
\sum_{k=1}^rL_k^*p_{2,k,n})\\
x_{n+1}=x_{n}-s_{1,1,n}+q_{1,1,n}\\[1mm]
\text{For}\;k=1,\ldots,r\\
\left\lfloor
\begin{array}{l}
y_{k,n+1}=y_{k,n}-p_{1,k+1,n}+q_{1,k+1,n}.
\end{array}
\right.\\
\end{array}
\right.\\
\end{array}
\end{equation}
Then the following hold for some solution $\overline{x}$ to 
\eqref{e:2013-05-08p} and some solution 
$(\overline{v_1},\ldots,\overline{v_r})$ to \eqref{e:2013-05-08d}.
\begin{enumerate}
\item
\label{t:2i}
$x_n\weakly\overline{x}~$ and 
$~(\forall k\in\{1,\ldots,r\})$ $v_{k,n}\weakly\overline{v_k}$.
\item
\label{t:2ii}
Suppose that $f$ or $\ell$ is uniformly convex at 
$\overline{x}$. Then $x_{n}\to\overline{x}$.
\item
\label{t:2iii}
Suppose that, for some $k\in\{1,\ldots,r\}$, $g^*_k$ 
is uniformly convex at $\overline{v_k}$. Then
$v_{k,n}\to\overline{v_k}$.
\end{enumerate}
\end{theorem}
\begin{proof}
Set
\begin{equation}
\label{e:2013-05-09a}
A=\partial f,\;C=\nabla\ell,\quad\text{and}\quad
(\forall k\in\{1,\ldots,r\})\quad
\begin{cases}
B_k=\partial g_k\\
D_k=\partial h_k. 
\end{cases}
\end{equation}
We derive from \cite[Proposition~17.10]{Livre1} that $C$ is 
monotone. Furthermore, 
\cite[Theorem~20.40 and Corollary~16.24]{Livre1}
assert that the operators $A$, $(B_k)_{1\leq k\leq r}$, and 
$(D_k)_{1\leq k\leq r}$ are maximally monotone with inverses
respectively given by $\partial f^*$, 
$(\partial g^*_k)_{1\leq k\leq r}$, and 
$(\partial h^*_k)_{1\leq k\leq r}$. Moreover, \eqref{e:R-stay} 
implies that \eqref{e:2013-04-30p} has a solution. Now let $x$ and 
$\boldsymbol{v}=(v_k)_{1\leq k\leq r}$ be,
respectively, the solutions to \eqref{e:2013-04-30p} and 
\eqref{e:2013-04-30d} produced by Theorem~\ref{t:1}.
Since the uniform convexity of a function at a point
implies the uniform monotonicity of its subdifferential at
that point \cite[Section~3.4]{Zali02} and since, in the setting of
\eqref{e:2013-05-09a}, \eqref{e:2013-05-08x} reduces to 
\eqref{e:2013-05-06x} thanks to \eqref{e:prox2}, 
it is enough to show that $x$ solves
\eqref{e:2013-05-08p} and $\boldsymbol{v}$ solves 
\eqref{e:2013-05-08d}. To this end, we first derive from 
\eqref{e:2013-05-18q} and 
\cite[Propositions~16.5(ii) and 24.27]{Livre1} that
\begin{align}
\label{e:2013-05-18b}
(\forall k\in\{1,\ldots,r\})\quad
\big(L_k^*\circ(\partial g_k)\circ L_k\big)\infconv
\big(M_k^*\circ(\partial h_k)\circ M_k\big)
&\subset\partial(g_k\circ L_k)\infconv\partial
(h_k\circ M_k)\nonumber\\
&=\partial\big((g_k\circ L_k)\infconv (h_k\circ M_k)\big).
\end{align}
Hence, it follows from \eqref{e:2013-05-09a} and Fermat's rule 
\cite[Theorem~16.2]{Livre1} that
\begin{align}
\label{e:2013-05-09b}
x~\text{solves \eqref{e:2013-04-30p}}
&\Rightarrow
z\in\partial f(x)+\sum_{k=1}^r\big((L_k^*\circ(\partial g_k)
\circ L_k)\infconv (M_k^*\circ(\partial h_k)\circ M_k)\big)x
+\nabla\ell(x)
\nonumber\\
&\Rightarrow
z\in\partial f(x)+\sum_{k=1}^r\partial\Big((g_k\circ L_k)
\infconv(h_k\circ M_k)\Big)x+\partial\ell(x)
\nonumber\\
&\Rightarrow
0\in\partial\Big(f+\sum_{k=1}^r\big((g_k\circ L_k)
\infconv(h_k\circ M_k)\big)+\ell-\scal{\cdot}{z}\Big)(x)
\nonumber\\
&\Rightarrow x~\text{solves \eqref{e:2013-05-08p}.}
\end{align}
On the other hand, \eqref{e:2013-05-17q} and 
Corollary~\ref{c:1}\ref{c:1ii} yield
\begin{equation}
\label{e:jeddah2013-05-19a}
(\forall k\in\{1,\ldots,r\})\quad
M_k^*\pushfwd\partial h_k^*=\partial(M_k^*\pushfwd h_k^*),
\end{equation}
while \cite[Proposition~16.5(ii)]{Livre1} yields
\begin{equation}
\label{e:jeddah2013-05-19b}
(\forall k\in\{1,\ldots,r\})\quad
\partial g_k^*+L_k\circ\big(\partial(M_k^*\pushfwd h_k^*)\big)
\circ L_k^*\subset\partial\big(g_k^*+(M_k^*\pushfwd h_k^*)
\circ L_k^*\big).
\end{equation}
Now define $\GGG$ as in \eqref{e:2013-05-16e} and
\begin{equation}
\label{e:jeddah2013-05-19c}
\begin{cases}
\varphi\colon\HH\to\RX\colon u\mapsto(f^*\infconv\ell^*)(z+u)\\
\boldsymbol{\psi}\colon\GGG\to\RX\colon
\boldsymbol{v}\mapsto\sum_{k=1}^r
\big(g_k^*(v_k)+(M_k^*\pushfwd h_k^*)(L_k^*v_k)\big)\\
\boldsymbol{M}\colon\GGG\to\HH\colon\boldsymbol{v}
\mapsto-\sum_{k=1}^rL_k^*v_k.
\end{cases}
\end{equation}
Then
\begin{equation}
\label{e:jeddah2013-05-19d}
(\forall\boldsymbol{v}\in\GGG)\quad
\varphi(\boldsymbol{M}\boldsymbol{v})
+\boldsymbol{\psi}(\boldsymbol{v})=(f^*\infconv\ell^*)
\bigg(z-\sum_{k=1}^rL_{k}^*v_k\bigg)+\sum_{k=1}^r
\big(g^*_k(v_k)+(M^*_k\pushfwd h^*_k)(L^*_kv_k)\big).
\end{equation}
Invoking successively \eqref{e:2013-05-09a}, 
\eqref{e:jeddah2013-05-19a}, \eqref{e:jeddah2013-05-19b}, 
\cite[Proposition~16.8]{Livre1},
\eqref{e:jeddah2013-05-19c}, \eqref{e:jeddah2013-05-19d}, 
and Fermat's rule, we get
\begin{align}
\label{e:2013-05-09c}
\boldsymbol{v}~\text{solves \eqref{e:2013-04-30d}}
&\Rightarrow
(\forall k\in\{1,\ldots,r\})\;\;
0\in-L_k\bigg(\partial(f+\ell)^*\bigg(z-\sum_{l=1}^rL_{l}^*
{v_l}\bigg)\bigg)\nonumber\\
&\quad\;+\partial g_k^*(v_k)+L_k\Big(\big(M_k^*\pushfwd
\partial h_k^*\big)(L_k^*v_k)\Big)
\nonumber\\
&\Rightarrow
(\forall k\in\{1,\ldots,r\})\;\;
0\in-L_k\bigg(\partial(f^*\infconv\ell^*)
\bigg(z-\sum_{l=1}^rL_{l}^*
v_l\bigg)\bigg)\nonumber\\
&\quad\;+\partial\big(g_k^*+(M_k^*\pushfwd h_k^*)\circ 
L_k^*\big)(v_k)
\nonumber\\
&\Rightarrow
\boldsymbol{0}\in\big(\boldsymbol{M}^*\circ(\partial\varphi)
\circ\boldsymbol{M}\big)(\boldsymbol{v})+
\partial\boldsymbol{\psi}(\boldsymbol{v})
\nonumber\\
&\Rightarrow
\boldsymbol{0}\in\partial\big(\varphi\circ\boldsymbol{M}
+\boldsymbol{\psi}\big)(\boldsymbol{v})
\nonumber\\
&\Rightarrow
\boldsymbol{v}~
\text{solves \eqref{e:2013-05-08d},}
\end{align}
which completes the proof.
\end{proof}

Theorem~\ref{t:2} enables us to solve a new class of structured
minimization problems featuring both infimal convolutions 
and postcompositions. The special cases of this model which
arise in the area of image recovery \cite{Cham97,Setz11} 
initially motivated our investigation. Such applications are 
considered in the next section.

\section{Image restoration application}
\label{sec:5}

\subsection{Image restoration}

Proximal splitting methods were introduced in the field of image 
recovery in \cite{Smms05} for variational models of the form
\begin{equation}
\label{e:prob-1}
\minimize{x\in\HH}{f(x)+\ell(x)},
\end{equation}
where $f$ and $\ell$ are as in Problem~\ref{prob:2} (see 
\cite{Banf11} for recent developments in this application area).
In this section we show a full fledged implementation of 
the algorithm in Theorem~\ref{t:2} in the Euclidean setting 
$\HH=\RR^N$ which goes much beyond \eqref{e:prob-1}. 
For this purpose, we consider the problem of 
image restoration from a blurred image \cite{Andrews77}.
Imaging devices, such as cameras, microscopes, and telescopes,
distort the light field due to both optical imperfections and
diffraction; another source of blur is relative movement of the
scene and the device during the exposure, as happens when taking a
photo in low-light without a tripod or when a telescope observes
the stars with imperfect motion compensation. The effect is that
the recorded image is the convolution of the true scene with a
function known as the point-spread function. The resulting
convolution operator $T$ is called the blur operator.
 
The original $N$-pixel ($N=512^2$) image shown in 
Fig.~\ref{fig:1}\subref{subfig1} is degraded by a linear blurring 
operator $T$ associated with a 21-pixel long point-spread function 
corresponding to motion blur, followed by addition of a noise
component $w$. Images in their natural matrix form are converted to
vectors $x\in \RR^N$ by stacking columns together.
We write the coefficients of $x$ as 
$x=(\xi_i)_{1\leq i\leq N}$, but when we wish
to make use of the 2-dimensional nature of the image 
(as a $\sqrt{N} \times \sqrt{N}$ image), we use the 
convention $\xi_{i,j}=\xi_{(j-1)\sqrt{N}+i}$ for every
$i$ and $j$ in $\{1,\ldots,\sqrt{N}\}$, so that  
$i$ and $j$ refer to the row and column indices, respectively.
The degraded image 
\begin{equation}
\label{e:model}
\obs=T\overline{x}+w 
\end{equation}
is shown in Fig.~\ref{fig:1}\subref{subfig2}. 
The noise level is chosen to give $\obs$ a 
signal-to-noise ratio of 45~dB relative to 
$T\overline{x}$.
The variational formulation we propose to recover 
$\overline{x}$ is an instantiation of Problem~\ref{prob:2}
with $r=2$, namely,
\begin{equation}
\minimize{x\in C}{
\big(
{(\alpha\|\cdot\|_{1,2}\circ \TV)}
\infconv
{(\beta\|\cdot\|_{1,2}\circ \higherTV )}
\big)(x)+{\gamma\|Wx\|_1}+
{\frac{1}{2}\|Tx-\obs\|_2^2}}
\end{equation}
or, equivalently,
\begin{multline}
\label{e:prob2}
\minimize{x\in\HH}{
\underbrace{\iota_C}_{f}(x)+\big(
\underbrace{(\alpha\|\cdot\|_{1,2}\circ \TV)}_{g_1 \circ L_1}
\infconv\underbrace{(\beta\|\cdot\|_{1,2}\circ
\higherTV )}_{h_1 \circ M_1}\big)(x)\\+
\big(\underbrace{\gamma\|W\cdot\|_1}_{g_2 \circ L_2}
\infconv
\underbrace{(\iota_{\{0\}}\circ\Id)}_{h_2 \circ M_2}\big)(x)+
\underbrace{\frac{1}{2}\|T\cdot-\obs\|_2^2}_{\ell}(x) }.
\end{multline}
In this model, $\alpha$, $\beta$, and $\gamma$ are strictly
positive constants, and $C$ is a constraint set modeling the 
known amplitude bounds on pixel values; here $C=[0,1]^{N}$. 
To promote the piecewise smoothness of $\overline{x}$ we use an
inf-convolution term mixing first- and second-order total variation
potentials, in a fashion initially advocated in \cite{Cham97}
and further explored in \cite{Setz11}. 
First-order total variation is commonly
used in image processing, but suffers from staircase effects (see,
e.g., \cite{Cham97}), which are reduced by using the
inf-convolution model. The operators $\TV$ and $\higherTV$
are, respectively, first and second order discrete gradient 
operators that map $\RR^N$ to $\RR^{N\times M}$ 
for $M=2$ and $M=3$, respectively
(see section~\ref{sec:TV} for details).
The functions $g_1$ and $h_1$ are the usual mixed
norms defined on $\RR^{N \times M}$ as 
\begin{equation}
\label{e:l1l2}
\|\cdot\|_{1,2}\colon x \mapsto \sum_{i=1}^N \sqrt{ \sum_{j=1}^M
\xi^2_{i,j} }.
\end{equation}
which is the sum of the norm of the rows of $x$.
The potential
\begin{equation}
x\mapsto\|Wx\|_1,
\end{equation}
where  $W$ is the analysis operator of a weighted 9/7 
biorthogonal wavelet frame~\cite{Cohe1992}, promotes sparsity of
wavelet coefficients of $x$. Since natural images are known to have
approximately sparse wavelet representations, this term penalizes
noise, which does not have a sparse wavelet representation.
Such wavelet terms are standard in the literature, and are often
used in conjunction with a first-order TV term~\cite{Pust11}.
Finally, data fidelity is promoted by the potential 
\begin{equation}
\ell\colon x\mapsto\frac12\|Tx-\obs\|^2. 
\end{equation}

\begin{remark}\label{rmk:1}
\rm
Here are some comments on the implementation of the algorithm
from Theorem~\ref{t:2} in the setting of \eqref{e:prob2}.
\begin{enumerate}
\item
The proximity operator of $f=\iota_C$ is simply the projector
onto a hypercube, which is straightforward.
\item
By \cite[Example~14.5]{Livre1}, for every 
$x\in\HH\smallsetminus\{0\}$, 
\begin{equation} 
\label{e:proxl1l2}
\prox_{\|\cdot\|} x=\left(1-\frac{1}{\|x\|}\right) x
\end{equation}
and $\prox_{\|\cdot\|} 0=0$.
Since $\|x\|_{1,2}$ is separable in the rows of $x$, $\prox_{
\|\cdot\|_{1,2} } x$ is computed by applying \eqref{e:proxl1l2} to
each row. 
\item
The gradient of $\ell$ is 
$\nabla\ell\colon x\mapsto T^\top(Tx-\obs)$, 
which is Lipschitz continuous with constant $\|T\|^2$.
\item
The proximity operator of $\|\cdot\|_1$ is implemented
by soft-thresholding of each component \cite{Banf11}.
\item
No special assumption is required on the structure of $W$
(e.g., the frame need not be tight or, in particular, an 
orthonormal basis). Without assumptions on $W$, there is no known
closed-form proximity operator of $x \mapsto \gamma \|Wx\|_1$, 
which is why it is important to treat $\|\cdot\|_1$ and 
$W$ separately.
\item
We have used only one hard constraint set $C$, but it is clear that
our framework can accommodate an arbitrary number of constraint
sets, hence permitting one to inject easily {\em a priori}
information in the restoration process. Each additional hard
constraint of the type $L_k\overline{x}\in C_k$ can be handled by
setting $g_k=\iota_{C_k}$, $h_k=\iota_{\{0\}}$, and $M_k=\Id$.
\end{enumerate}
\end{remark}

\begin{remark}
Remark~\ref{rmk:1} shows that the computation of proximity
operators for each function involved in~\eqref{e:prob2} is
implementable.  It is also possible to compute proximity operators
for scaled versions of the above functions. Let $\rho\in\RPP$. 
Then given $\varphi\in\Gamma_0(\HH)$ and 
$\widetilde{\varphi}\colon x\mapsto\varphi(\rho x)$, 
\cite[Corollary~23.24]{Livre1} implies that
\begin{equation}
(\forall x\in\HH)\quad\prox_{\widetilde{\varphi}}\,
x=\rho^{-1}\prox_{\rho^2 \varphi} (\rho x).
\end{equation}
This gives the possibility of writing $f(Lx)$ as
$\tilde{f}(\tilde{L}x)$ for $\tilde{L}=\rho^{-1}L$. Our
implementation will exploit this flexibility in order to rescale
all $L_k$ and $M_k$ operators to have unit operator norm.
Numerical evidence suggests that this improves convergence profiles
since all dual variables $(v_k)_{1\leq k\leq r}$ and 
$(w_k)_{1\leq k\leq r}$ are approximately of the same scale.
\end{remark}

\subsection{Total variation}
\label{sec:TV}
Total variation can be defined for mathematical objects such as
measures and functions \cite{Ziem89}. In a discrete setting, 
there are many possible definitions of total variation.
We use the standard isotropic discretization,
\begin{equation}
\label{e:TV}
\tv(x)= \sum_{i=1}^{\sqrt{N}-1}\sum_{j=1}^{\sqrt{N}-1} \sqrt{ 
\left( \xi_{i+1,j} - \xi_{i,j} \right)^2 + 
\left( \xi_{i,j+1} - \xi_{i,j} \right)^2 }, \quad
x=(\xi_k)_{1\leq k \leq N}, \,\xi_{i,j}=\xi_{(j-1)\sqrt{N}+i},
\end{equation}
originally advocated in \cite{Rudi92}. There is no known closed
form expression for the proximity operator of \eqref{e:TV}. 

Infimal-convolution with a second-order total variation term was
first suggested in~\cite{Cham97}. We use the 
particular second-order
total variation term corresponding to ``$\mathcal{D}_{2,b}$'' (with
weights $b=(1,\frac{1}{2},1)$) from~\cite{Setz11}.  
We now show how to recover the relation $\tv(x)=\|\TV x\|_{1,2}$.
Define the following horizontal finite-difference operator 
\begin{equation} 
\label{e:tvdiff1}
\Dh\colon\RR^{N} \rightarrow \RR^{\sqrt{N} \times \sqrt{N}} 
\colon x \mapsto z = (\zeta_{i,j})_{1\leq i,j \leq \sqrt{N}}, 
\;\zeta_{i,j}=
\begin{cases} 
\xi_{i,j+1} - \xi_{i,j} & 1 \leq j < \sqrt{N} \\
0 & j=\sqrt{N}, 
\end{cases}
\end{equation}
and define the  vertical operator $\Dv$ by
$\Dv\colon \; x \mapsto (\Dh(x^\top))^\top$.
Let $\VEC(\cdot)$ be the
mapping that re-orders a matrix by stacking the columns together,
and define $\TV\colon x\mapsto( \VEC(\Dh(x)), \VEC(\Dv(x)) )$.
Then by comparing \eqref{e:l1l2} with \eqref{e:TV}, we 
observe that $\tv(x)=\|\TV x\|_{1,2}$.

The second-order total variation potential makes use of an
additional set of first-order difference operators that have
different boundary conditions, namely
\begin{equation}
\label{e:tvdiff2}
\Dhh: \RR^{N} \rightarrow \RR^{\sqrt{N} \times \sqrt{N}} 
\colon x \mapsto z = (\zeta_{i,j})_{1\leq i,j \leq \sqrt{N}},
\;\zeta_{i,j}=
\begin{cases}
\xi_{i,j} - \xi_{i,j-1} & 1 <j < \sqrt{N} \\
\xi_{i,j} & j=1 \\
-\xi_{i,j-1} & j=\sqrt{N} \end{cases},
\end{equation}
and $\Dvv\colon x \mapsto (\Dhh(x^\top))^\top$.
Then define 
\begin{equation}
\higherTV\colon x\mapsto
\left(\VEC(\Dhh(\Dh x)), \frac{\VEC( \Dhh(\Dv x)) +
\VEC(\Dvv(\Dh x))}{\sqrt{2}}, \VEC( \Dvv(\Dv x) ) \right).
\end{equation}
The second-order total variation potential is
defined as $x\mapsto\|\higherTV x\|_{1,2}$.

\subsection{Constraint qualifications}
To apply the results of Theorem~\ref{t:2}, we need to check that 
the constraint qualifications \eqref{e:2013-05-12q},
\eqref{e:2013-05-17q}, and \eqref{e:R-stay} hold.
Starting with \eqref{e:2013-05-12q}, for each $k\in\{1,2\}$ 
we have
\begin{align} 
\label{e:42}
\sri\big(\dom (g_k \circ L_k)^*-M_k^*(\dom h_k^*)\big) 
&=\sri\big(\dom ( L_k^*\pushfwd g_k^*)-M_k^*(\dom h_k^*) \big) 
\nonumber \\
&=\sri\big( L_k^*(\dom g_k^*)-M_k^*(\dom h_k^*)\big) 
\nonumber \\
&=L_k^*\big(\reli \dom g_k^*)-M_k^*(\reli \dom h_k^*\big),
\end{align}
where the first line follows from \cite[Proposition~15.28]{Livre1}
and the fact that $g_k$ has full domain, the
second line follows from \cite[Proposition~12.34(i)]{Livre1}, and
the third line follows from \cite[Corollary~6.15]{Livre1}. Since
$g_1$, $g_2$, and $h_1$ are coercive, their 
conjugates all include $0$ in the interior of their domain
\cite[Theorem~14.17]{Livre1}. Furthermore, the conjugate of
$h_2=\iota_{\{0\}}$ is $h_2^* = 0$ which has full domain. Thus,
\begin{equation}
(\forall k\in\{1,2\})\quad 0 \in L_k^*(\reli \dom g_k^*)
\quad\text{and}\quad 0\in M_k^*(\reli \dom h_k^*).
\end{equation}
Altogether, \eqref{e:42} is satisfied for each $k\in\{1,2\}$ and 
hence so is \eqref{e:2013-05-12q}.
The qualification \eqref{e:2013-05-17q} holds for $k=1$ since
$h_1=\|\cdot\|_{1,2}$ has full domain. For $k=2$, since
$h_2=\iota_{\{0\}}$, using \cite[Corollary~6.15]{Livre1} 
and the linearity of $M_2$, we obtain
\begin{equation}
\sri(\ran M_2-\dom h_2)=\sri( \ran M_2 )=\reli( \ran M_2 ) 
=\ran M_2.
\end{equation}
Thus, since $0\in\ran M_2$, \eqref{e:2013-05-17q} is satisfied.
On the other hand, 
since $\HH$ is finite-dimensional, the constraint qualifications
\eqref{e:R-stay} is implied by
Proposition~\ref{p:2013-05-17b}\ref{p:2013-05-17bv}.
Both $g_1$ and $h_1$ are norms and therefore have full domain, so 
\eqref{e:cq2} is satisfied for $k=1$. For $k=2$, 
$g_2$ is a norm and has full domain while 
$h_2=\iota_{\{0\}}$ so $0 \in \reli \dom h_2$ and hence
\eqref{e:cq2} holds for $k=2$.

To apply Proposition~\ref{p:2013-05-17b}, the primal problem must
have a solution. Here existence of a solution follows from the
compactness of $C$ \cite[Section~11]{Livre1}. 

\begin{figure}[t]
\centering
\subfloat[Original image]{
\includegraphics[width=\mySubFigSize]%
{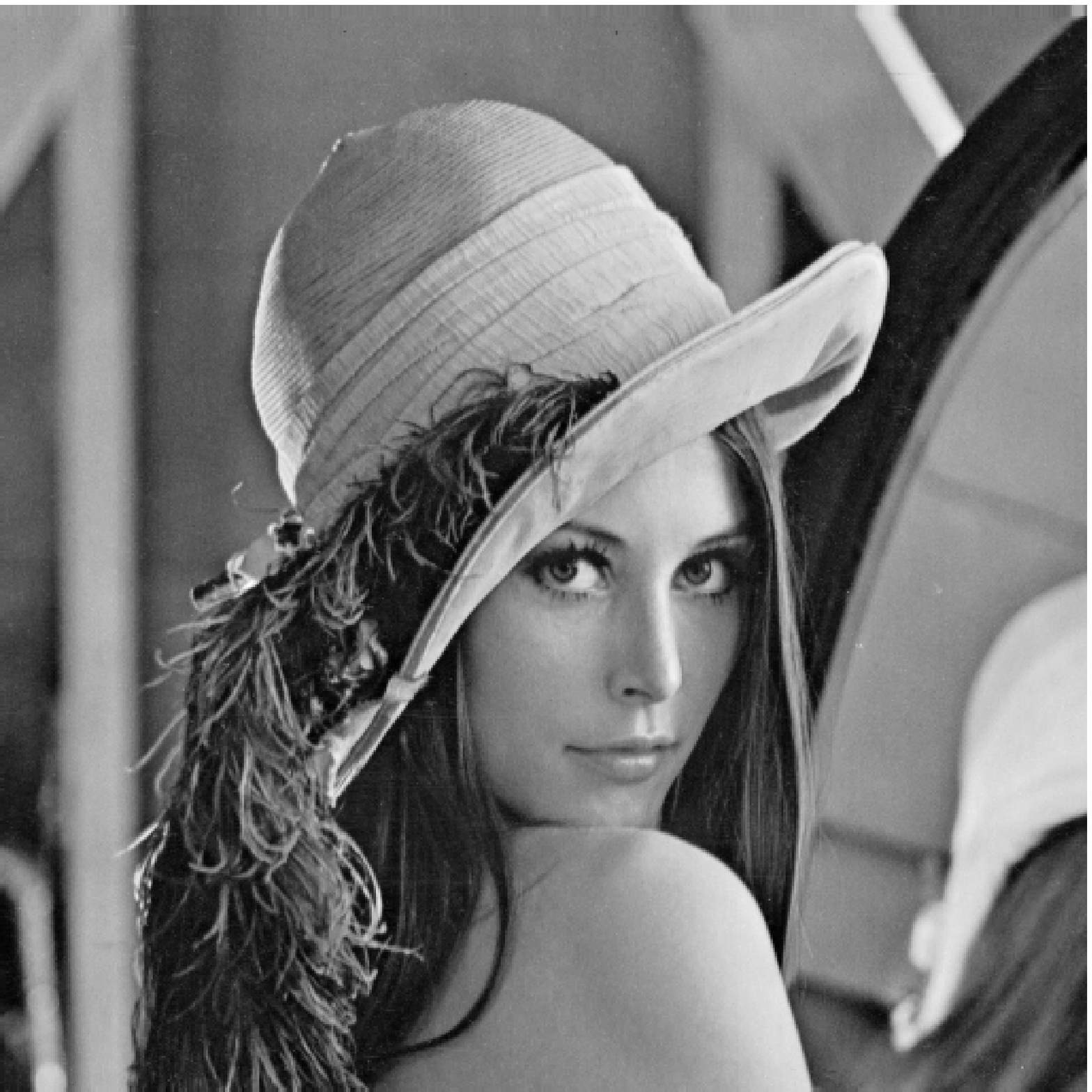}\label{subfig1}}
\subfloat[Degraded image]{
\includegraphics[width=\mySubFigSize]%
{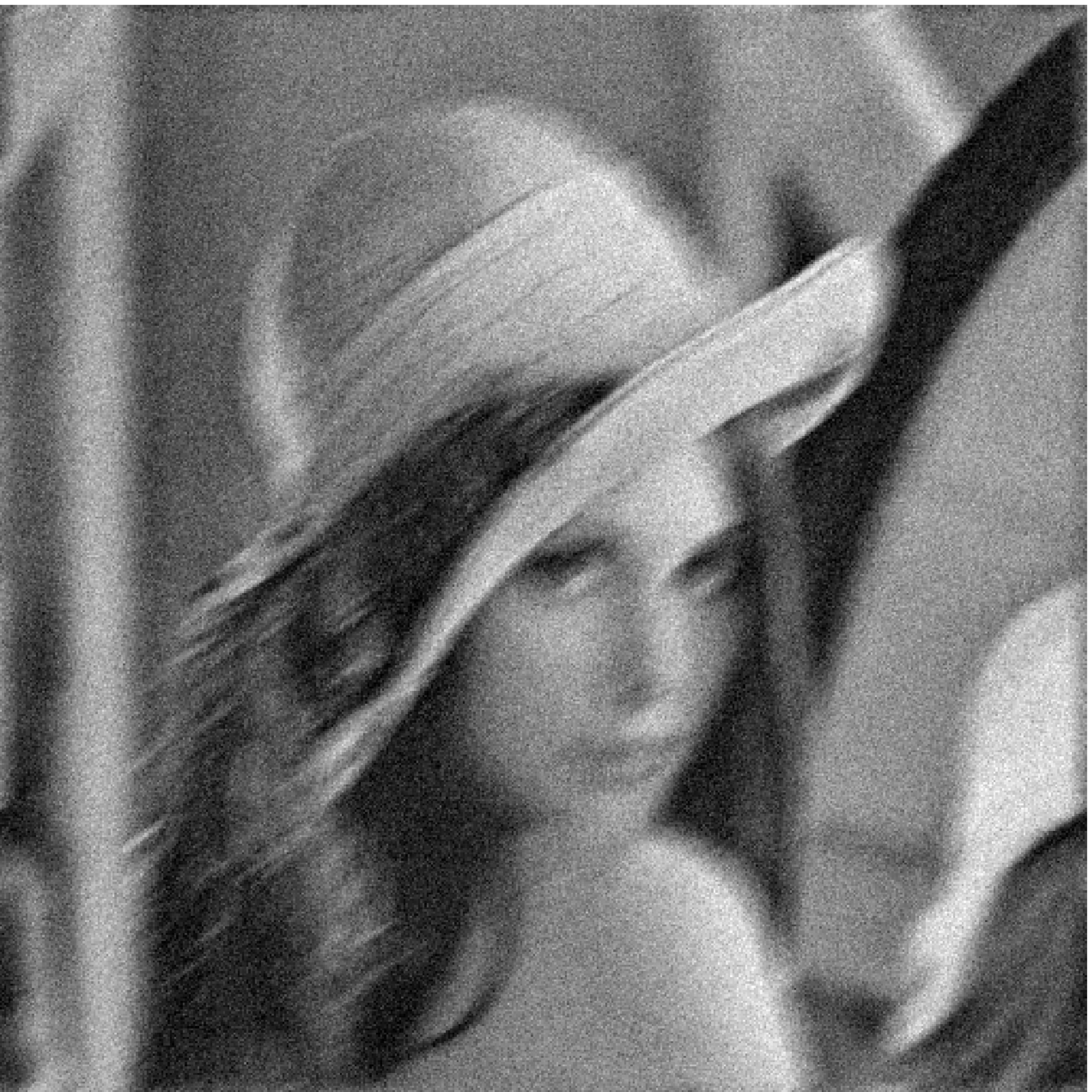}\label{subfig2}}
\subfloat[Solution to \eqref{e:prob2}]{
\includegraphics[width=\mySubFigSize]%
{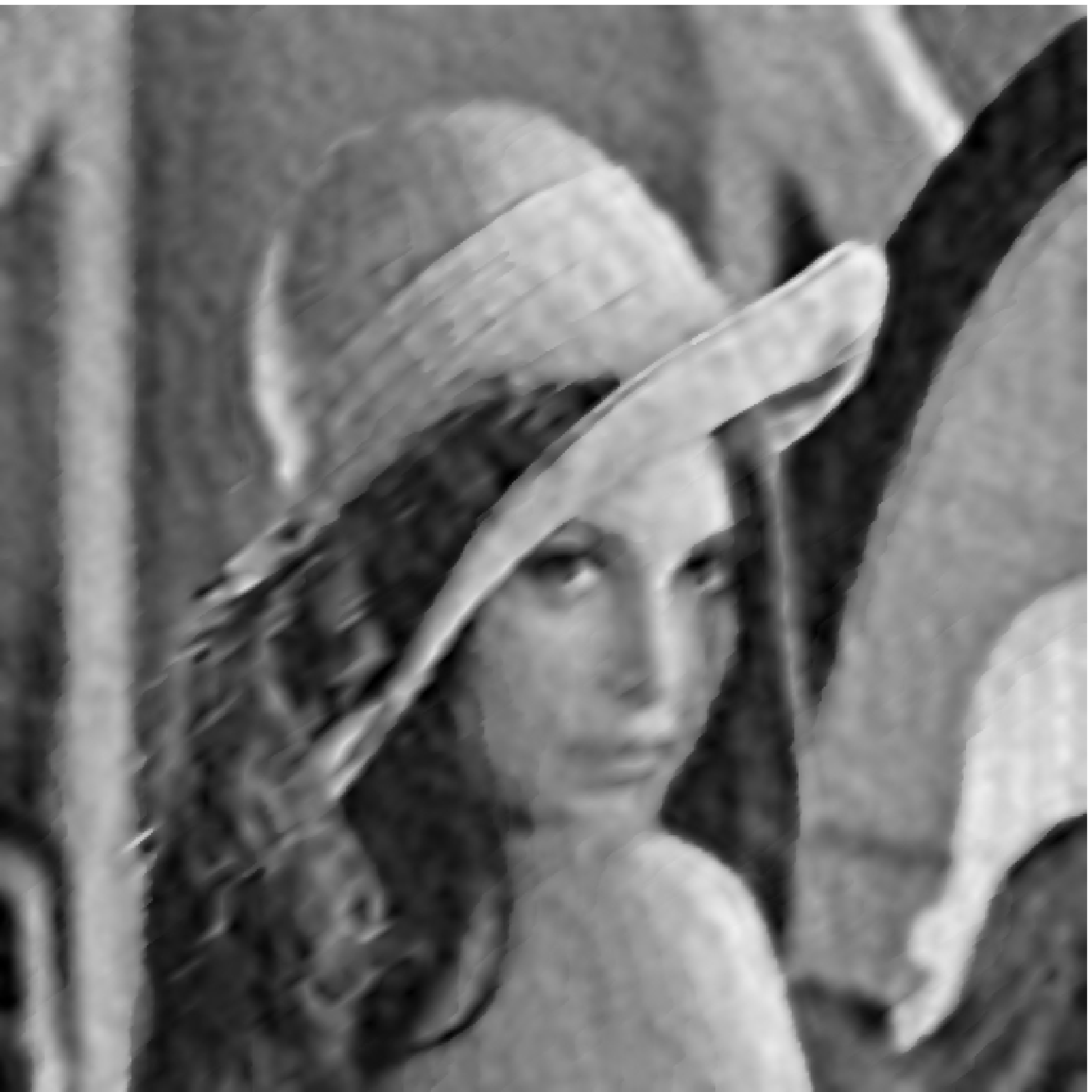}\label{subfig3}}
\caption{Original, blurred, and restored images.
\label{fig:1}}
\end{figure}

\subsection{Numerical experiments}

Experiments are made on a quad-core 1.60 GHz Intel i7 laptop, with
the algorithms and analysis implemented using the free software
package GNU Octave \cite{Octave}. The authors are grateful 
for the support of the Octave development community.

Note that in \eqref{e:2013-05-08x}, the update for $s_{1,1,n}$ and
for $p_{1,k+1,n}$ both involve $L_{k}^*v_{k,n}$,
hence it is possible to prevent redundant computation by storing
$L_{k}^*v_{k,n}$ as a temporary variable.  Similarly, the updates
for $q_{1,1,n}$ and $q_{1,k+1,n}$ both involve $L_{k}^*p_{2,k,n}$,
which can also be stored as a temporary variable for savings. With
this approach, each $L_{k}$ and $M_{k}$ is applied exactly twice
per iteration, and each  $L_{k}^*$ and $M_{k}^*$ is also applied
exactly twice.
The restored image is displayed in
Fig.~\ref{fig:1}\subref{subfig3}.  The algorithm uses all variables
initialized to $0$. The values of the parameters are as follows:
$\alpha=\beta=\gamma=10^{-2}$. 
Figures of merit relative to these experiments are provided in
Table~\ref{tab:1}. Given a reference image $x$ and an estimate
$\overline{x}=(\overline{\xi}_i)_{1\leq i\leq N}$, the 
peak signal-to-noise ratio
(PSNR), a standard measure of image quality, is defined by
\begin{equation}
\text{PSNR}_{x}(\overline{x})=10\log_{10}\left( \frac{
N\max_{1\leq i\leq N}\xi_{i}^2}{
\sum_{i=1}^N(\xi_{i}-\overline{\xi}_{i})^2}\right)
\end{equation}
and reported in units of decibels (dB). The structural similarity
index attempts to quantify human visual response
to images; details can be found in \cite{SSIM}.

\begin{table}
\centering
\caption{Quantitative measurements of performance}
\begin{tabular}{lll}
\toprule
Method & Peak signal-to-noise ratio & Structural 
similarity index \\
\midrule
Blurred and noisy image & 20.32 dB  & 0.545 \\
Restoration          & 25.42 dB & 0.803 \\
\bottomrule
\end{tabular}
\label{tab:1}
\end{table}

{\bfseries Acknowledgment.}
The authors thank J.-C.~Pesquet for the implementation of the
wavelet transform. S. R. Becker is funded by the Fondation Sciences
Math\'ematiques de Paris.  The research of P. L. Combettes is
supported in part by the European Union under the 7th Framework
Programme ``FP7-PEOPLE-2010-ITN'', grant agreement number
264735-SADCO.

\end{document}